\documentclass[
    a4paper,
    11pt,
]{article}

\usepackage{amsmath,amssymb,amsthm}
\usepackage{mathtools,dsfont}
\usepackage{paralist,enumitem,bm,placeins,url}
\usepackage{color,graphicx,subfig} 
\usepackage{cite}

\usepackage{tabularx}
\usepackage{booktabs}
\usepackage{tikz,pgfplots,tikz-3dplot}
\usetikzlibrary{fit}

\usepackage{array}
\newcolumntype{C}[1]{>{\centering\arraybackslash}p{#1}}
\usepackage{tocbasic}
\DeclareTOCStyleEntry[
beforeskip=.2em plus 1pt,
]{tocline}{section}
\usepackage{todonotes}

\usepackage{hyperref} 
\makeatletter
\hypersetup{%
	colorlinks	=true,
	linkcolor	=blue,%
	citecolor	=red,%
	urlcolor	=blue
}

\usepackage[normalem]{ulem}
\usepackage{mathrsfs}
\usepackage{ulem}
\usepackage{caption}
\usepackage{enumitem}
\pgfplotsset{compat=1.7}
\usepackage[title]{appendix}

\newtheorem{thm}{Theorem}[section]
\newtheorem{lem}[thm]{Lemma}
\newtheorem{rem}[thm]{Remark}

\newcommand*{\vv}[1]{\vec{\mkern0mu#1}}
\newcommand{\ek}{e}

\newcommand{\norm}[1]{\Vert#1\Vert}
\newcommand{\bRplus}{\bR_{>0}}
\newcommand{\bRgeq}{\bR_{\geq 0}}

\newcommand{\pV}{\underline{V}_\partial}

\newcommand{\bR}{{\mathbb R}}

\newcommand{\bZ}{\mathbb{Z}}
\newcommand{\bI}{\mathbb{I}}

\newcommand{\bkap}{{\widebar{\varkappa}}}
\newcommand{\mS}{{\mathcal{S}}}
\newcommand{\mV}{{\mathscr{V}}}
\newcommand{\ipd}[1]{\Bigl(#1\Bigr)}

\def\widebar{\overline}

\newcommand{\rd}{\;{\rm d}}
\newcommand{\id}{{\rm id}}
\newcommand{\dd}[1]{\frac{\rm d}{{\rm d}#1}}
\newcommand{\ddt}{\dd{t}}
\newcommand{\nn}{\nonumber}
\newcommand{\ttau}{\Delta t}

\numberwithin{equation}{section}

{{\upshape\bfseries AMS subject classifications. }\ignorespaces}{}

\begin{document}
\title{
Stable fully practical finite element methods for \\
axisymmetric Willmore flow 
}

\author{Harald Garcke\footnotemark[1] 
    \and Robert N\"urnberg\footnotemark[2]
    \and Quan Zhao\footnotemark[3]}

\renewcommand{\thefootnote}{\fnsymbol{footnote}}

\footnotetext[1]{Fakult\"at f\"ur Mathematik, Universit\"at Regensburg, 93053 Regensburg, Germany, \\ \tt(\href{mailto:harald.garcke@ur.de}{harald.garcke@ur.de})}
	
\footnotetext[2]{Dipartimento di Mathematica, Universit\`a di Trento,
38123 Trento, Italy, \\
\tt(\href{mailto:robert.nurnberg@unitn.it}{robert.nurnberg@unitn.it}) }

\footnotetext[3]{School of Mathematical Sciences, University of Science and Technology of China, 230026 Hefei,  China \\
\tt(\href{mailto:quanzhao@ustc.edu.cn}{quanzhao@ustc.edu.cn}) }

\date{}
\maketitle

\begin{abstract}
\noindent
We consider fully discrete numerical approximations for axisymmetric 
Willmore flow that are unconditionally stable and
work reliably without remeshing. We restrict our attention
to surfaces without boundary, but allow for spontaneous curvature effects. 
The axisymmetric setting allows us to formulate our schemes in terms of 
the generating curve of the considered surface. 
We propose a novel weak formulation, that combines an evolution equation for 
the surface's mean curvature and the curvature identity of the generating 
curve. The mean curvature is used to describe the gradient flow structure, 
which enables an unconditional stability result for the discrete solutions. 
The generating curve's curvature, on the other hand, describes the surface's
in-plane principal curvature and plays the role of a 
Lagrange multiplier for an equidistribution property on the discrete level. 
We introduce two fully discrete schemes and prove their unconditional 
stability. Numerical results are provided to confirm the convergence, stability and equidistribution properties of the introduced schemes.  
\end{abstract}



\setlength\parskip{1ex}
\renewcommand{\thefootnote}{\arabic{footnote}}

\setcounter{equation}{0}

\section{Introduction}
\setlength\parindent{24pt}

The Willmore energy, which measures the bending of a surface $\mS$, is given by the surface integral of the square of the 
mean curvature $\varkappa_{_\mS}$ (here the sum of the principle curvatures)
over the surface \cite{Willmore93}
\begin{equation}\label{eq:WillmoreE1}
\frac{1}{2}\int_{\mS}\varkappa_{_\mS}^2\,\rd A. 
\end{equation}
The minimization of the Willmore energy leads to so-called Willmore surfaces. Spheres are the well-known Willmore surfaces of genus 0, as their energies attain the minimum of $8\pi$. The Clifford torus has only recently been proved to be the minimizer among surfaces in $\bR^3$ of genus 1, with Willmore energy of $4\pi^2$ \cite{Marques14min}. Broad applications involving the Willmore energy have been found in the fields of biomembranes \cite{Canham1970minimum, Helfrich73elastic,Seifert97}, computer graphics \cite{Welch94free,Desbrun99implicit}, materials science \cite{Jiang16solid,ChenLSWW18,Grinspun08}, 
computational geometry \cite{BretinLO11,GruberA20} and surface
restoration \cite{ClarenzDDRR04,LeeC23}.

Let $\mS$ be a closed surface in $\bR^3$. We define the mean curvature $\varkappa_{_\mS}$ via
\begin{equation}\label{eq:WillmoreE0}
\varkappa_{_\mS}\,\vec n_{_\mS} = \Delta_{_\mS}\vec\id\quad\mbox{on}\quad \mS,
\end{equation}
where $\vec n_{_\mS}$ is the unit normal to $\mS$,  $\Delta_{_\mS}=\nabla_{_\mS}\cdot\nabla_{_\mS}$ is the Laplace-Beltrami operator with $\nabla_{_\mS}$ being the surface gradient, and $\vec\id$ is the identity function on $\bR^3$. Observe that our sign convention is such that the unit sphere with outer normal $\vec n_{_\mS}$ has mean curvature $\varkappa_{_\mS} = -2$.
A natural bending energy for a surface $\mS$ representing a membrane is then
given by the Willmore energy \eqref{eq:WillmoreE1}. 
In biomembranes, a different chemical environment on the two sides of the membrane leads to an asymmetry in the elastic forces. Then an energy including spontaneous curvature effects can be introduced as \cite{Helfrich73elastic}
\begin{equation}\label{eq:energyS}
\frac{1}{2}\int_{\mS}(\varkappa_{_\mS} - \bkap)^2\rd A, 
\end{equation}
where $\bkap\in\bR$ is the so-called spontaneous curvature which reflects the asymmetry. As before, ${\rm d} A$ denotes integration with respect to the surface measure.
The resulting $L^2$-gradient flow of \eqref{eq:energyS} is then given by 
\cite{BGN08willmore, Kuwert01willmore, Simonett05willmore}
\begin{align}
\mathcal{V}_{_\mS}  &=-\Delta_{_\mS}\varkappa_{_\mS}-(\varkappa_{_\mS}-\bkap)|\nabla_{_\mS}\vec n_{_\mS}|^2 + \frac{1}{2}(\varkappa_{_\mS} - \bkap)^2\varkappa_{_\mS} \nn\\
&=  -\Delta_{_\mS}\varkappa_{_\mS} + 2(\varkappa_{_\mS}-\bkap)\mathcal{K}_{_\mS}-\frac{1}{2}(\varkappa_{_\mS}^2-\bkap^2)\varkappa_{_\mS}\qquad\mbox{on}\quad\mS, \label{eqn:WillmoreS}
\end{align}
where $\mathcal{V}_{_\mS}$ is the normal velocity of an evolving surface $\mS(t)$, $\nabla_{_\mS}\vec n_{_\mS}$ is the Weingarten map and $\mathcal{K}_{_\mS}$ stands for the Gauss curvature. The gradient flow \eqref{eqn:WillmoreS} is a fourth-order geometric partial differential equation (PDE) and obeys the energy dissipation law 
\begin{equation}
\ddt\left[\frac{1}{2}\int_{\mS(t)}(\varkappa_{_\mS} -\bkap)^2\rd A\right] = -\int_{_{\mS(t)}}[\mathcal{V}_{_{\mS}}]^2\rd A\leq 0.\label{eq:WillED}
\end{equation}
For analytical results of Willmore flow, we refer the reader to \cite{Kuwert01willmore, DziukKS02, KS02, Simonett05willmore, Blatt09singular, Dall24willmore,Schlierf25spont} and references therein. From the numerical analysis point of view, it is of great interest to design numerical methods for the flow \eqref{eqn:WillmoreS} such that gradient flow structure \eqref{eq:WillED} can be mimicked on the fully discrete level. 

In this work, we will focus on parametric approximations of Willmore flow, and for other numerical approaches, we refer the reader to the review articles \cite{DeckelnickDE05,Barrett20} and references therein.  Earlier parametric approximations for Willmore flow have been considered in the context of curves \cite{DziukKS02} and surfaces in \cite{Rusu05}. With the help of the (mean) curvature vector, both methods use mixed formulations to allow piecewise linear elements for the spatial discretization, but fail to satisfy energy stability estimates on the discrete level. Later Dziuk \cite{Dziuk08} derived a new weak formulation which leads to a stable semidiscrete finite element approximation.  Recently, Kov\'acs, Li and Lubich considered a new formulation for Willmore flow by including the evolution of geometric quantities \cite{KovacsLL19}. This allows for a convergence proof of a finite element method for Willmore flow of surfaces.  Nevertheless, all of the aforementioned approaches discretize a purely normal flow, which often leads to coalescence of vertices and degenerate surface meshes.

A natural remedy to prevent mesh degenerations during simulations is to incorporate appropriate tangential velocity in the flow, see \cite{BGN07,Hu22evolving,DeTurck17,Duan24new}. Barrett, Garcke, and N\"urnberg (BGN) introduced a weak formulation that allows tangential degrees of freedom so that the mesh quality can be improved. In particular, the BGN tangential motion leads to an equidistribution of mesh points \cite{BGN07} on the semidiscrete level for the evolving polygonal curves. The interested reader is referred to the recent review paper \cite{Barrett20} for a detailed discussion. The BGN approach was employed for Willmore flow of curves\slash surfaces in \cite{BGN07,BGN08willmore,curves3d}, but while the methods are practical, no stability estimates are available. Building upon the original idea in \cite{Dziuk08}, adapted BGN schemes
were introduced in \cite{pwf,pwfade,pwfopen} which preserve the energy dissipation property on the semidiscrete level. Inspired by work in \cite{KovacsLL19}, recently Bao and Li introduced a novel numerical scheme for the Willmore flow of planar curves, which achieves energy stability for fully discrete parametric approximations \cite{BaoL25}. The new scheme from \cite{BaoL25} incorporates tangential motion into its framework; however it may still require mesh redistributions for complex curve evolutions. More recently, the present authors combined the idea from \cite{BaoL25} and the tangential motion from the BGN approach to obtain fully discrete finite element approximations that satisfy an unconditional stability estimate in terms of the energy, and in addition satisfy an equidistribution property for the mesh points of the polygonal curves \cite{GNZ25willmore}.  

It is the main aim of this paper to adapt the idea in \cite{GNZ25willmore} and devise fully finite element approximations for the axisymmetric Willmore flow of surfaces. The axisymmetric setting reduces the geometric flow into a one-dimensional problem by simply dealing with the generating curve of the axisymmetric surfaces. There are several existing numerical works for axisymmetric Willmore flow. One of the earliest contributions is the finite difference approximation employed by Mayer and Simonett in \cite{Mayer02numerical}, which provides numerical evidence for the formation of singularities in Willmore flow. The BGN formulation was also considered for axisymmetric surfaces \cite{BGN19asy,BGN21stable}, and a stability result can be shown on the semidiscrete level in \cite{BGN21stable}. Parametric approximations for axisymmetric surfaces based on the work in \cite{BaoL25} were analyzed in the recent works \cite{Ma25energy,LL25axis}, but these still suffer from strong mesh distortions. For example, they appear to be unable to accurately capture the stationarity of the Clifford torus. 

In the current work, we would like to incorporate our recent idea from \cite{GNZ25willmore} and explore parametric finite element approximations for  Willmore flow of closed surfaces in the axisymmetric setting. We will also include spontaneous curvature effects and consider the geometric flow \eqref{eqn:WillmoreS} in the case when $\mS(t)$ is an axisymmetric surface that is rotationally symmetric with respect to the $x_2$–axis, as shown in Fig.~\ref{fig:sketch}. The key idea of our approach is to introduce a new geometric PDE for the axisymmetric Willmore flow of surfaces. This relies on an evolution equation for the mean curvature of the surface and a curvature formulation of the generating curve. In particular, the first curvature discretization approximates the gradient flow structure and allows for an unconditional stability estimate on the discrete level. Discretization of the second curvature not only helps to compute the Gauss curvature but also introduces the desired BGN tangential motion to enable an asymptotically equidistribution property of mesh points on the discrete curves. 

The rest of the paper is organized as follows. We begin in Section~\ref{sec:mf} with the axisymmetric setting. There we introduce a new geometric PDE for Willmore flow of axisymmetric surfaces and propose a novel weak formulation for it. Next in Section~\ref{sec:semi},  we consider a semidiscrete approximation of the weak formulation and show that the approximation satisfies a stability estimate and an equidistribution property. Subsequently in Section~\ref{sec:fulld} we devise two fully discrete finite element schemes and prove their unconditional energy stability. Numerical examples are reported in Section~\ref{sec:numr} to examine the convergence, stability and good mesh properties of the introduced schemes. Finally we draw some conclusions in Section~\ref{sec:con}.

\section{Mathematical formulation}
\label{sec:mf}

\subsection{The axisymmetric setting}

\begin{figure}[!htp]
\center
\newcommand{\AxisRotator}[1][rotate=0]{%
    \tikz [x=0.25cm,y=0.60cm,line width=.2ex,-stealth,#1] \draw (0,0) arc (-150:150:1 and 1);%
}
\begin{tikzpicture}[every plot/.append style={very thick}, scale = 1]
\begin{axis}[axis equal,axis line style=thick,axis lines=center, xtick style ={draw=none}, 
ytick style ={draw=none}, xticklabels = {}, 
yticklabels = {}, 
xmin=-0.2, xmax = 0.8, ymin = -0.4, ymax = 2.55]
after end axis/.code={  
   \node at (axis cs:0.0,1.5) {\AxisRotator[rotate=-90]};
   \draw[blue,->,line width=2pt] (axis cs:0,0) -- (axis cs:0.5,0);
   \draw[blue,->,line width=2pt] (axis cs:0,0) -- (axis cs:0,0.5);
   \node[blue] at (axis cs:0.5,-0.2){$\vec\ek_1$};
   \node[blue] at (axis cs:-0.2,0.5){$\vec\ek_2$};
   \draw[red,very thick] (axis cs: 0,0.7) arc[radius = 70, start angle= -90, end angle= 90];
    \draw[blue,->,line width=1.2pt] (axis cs:0.7*0.950, 1.3+0.5*0.7) -- (axis cs:1.15*0.950,1.3+0.5*1.15);
   \node[blue] at (axis cs:0.96,1.65){$\vec\nu$};
   \node[red] at (axis cs:0.7,1.9){$\Gamma$};
}
\end{axis}
\end{tikzpicture} \qquad \qquad
\tdplotsetmaincoords{120}{50}
\begin{tikzpicture}[scale=2, tdplot_main_coords,axis/.style={->},thick]
\draw[axis] (-1, 0, 0) -- (1, 0, 0);
\draw[axis] (0, -1, 0) -- (0, 1, 0);
\draw[axis] (0, 0, -0.2) -- (0, 0, 2.7);
\draw[blue,->,line width=2pt] (0,0,0) -- (0,0.5,0) node [below] {$\vec\ek_1$};
\draw[blue,->,line width=2pt] (0,0,0) -- (0,0.0,0.5);
\draw[blue,->,line width=2pt] (0,0,0) -- (0.5,0.0,0);
\node[blue] at (0.2,0.4,0.1){$\vec\ek_3$};
\node[blue] at (0,-0.2,0.3){$\vec\ek_2$};
\node[red] at (0.7,0,1.9){$\mathcal{S}$};
\node at (0.0,0.0,2.4) {\AxisRotator[rotate=-90]};

\tdplottransformmainscreen{0}{0}{1.4}
\shade[tdplot_screen_coords, ball color = red] (\tdplotresx,\tdplotresy) circle (0.7);
\end{tikzpicture}
\caption{Sketch of $\Gamma$ and $\mathcal{S}$, as well as 
the unit vectors $\vec\ek_1$, $\vec\ek_2$ and $\vec\ek_3$.}
\label{fig:sketch}
\end{figure}
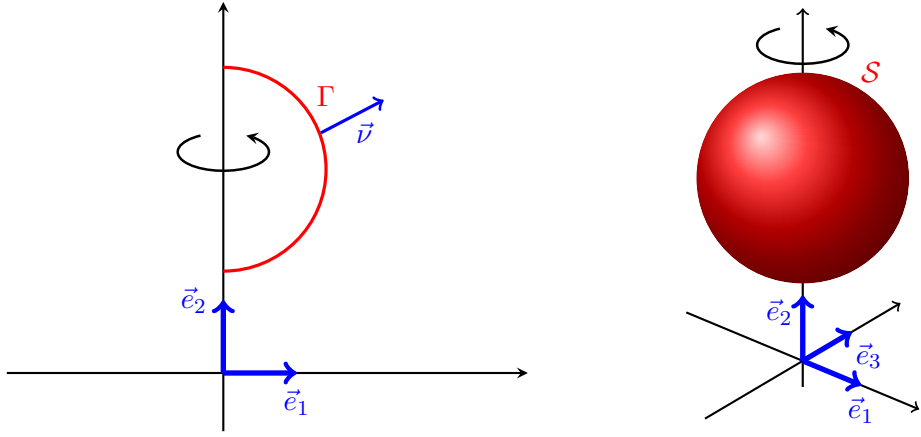

We consider the axisymmetric case and assume that the surface $\mS(t)$ satisfies the rotational symmetry with respect to the $x_2$-axis, as shown in Fig.~\ref{fig:sketch}. In what follows our presentation will closely follow the notation introduced in \cite{BGN19asy}. 
Denote by $\Gamma(t)$ the generating curve of $\mS(t)$, and its parameterization over $\mathbb{I}$ is given by 
\begin{equation}
\vec x(\rho,t)= \bigl(r(\rho,t),~z(\rho,t)\bigr)^T:\bar{\mathbb{I}}\times [0,T]\mapsto\bRgeq\times\bR,\nn
\end{equation}
where the reference domain $\mathbb{I}$ is either the periodic unit interval or the unit interval with boundary: 
\[\mathbb{I} = \bR\slash\bZ\quad\mbox{with}\quad \partial\mathbb{I} = \emptyset,\qquad\mbox{or}\qquad \mathbb{I} = (0,1)\quad\mbox{with}\quad \partial\mathbb{I}=\{0,~1\}.\]
Then an induced parameterization of the axisymmetric surface is given by
\begin{equation}
(\rho,\theta, t)\mapsto\vec y(\rho,\theta, t) =\bigl(r(\rho,t)\cos\theta,~z(\rho,t),~r(\rho,t)\sin\theta\bigr)^T,\quad \rho\in \bar{\bI}, \quad \theta\in[0,2\pi],\nn
\end{equation}
for $t\in[0,T]$, and $\theta$ is the azimuthal angle. Here, we allow the generating curve $\Gamma(t)$ to be either a closed curve, parameterized over $\bR/\bZ$, which corresponds to $\mS(t)$ being a genus-1 surface without boundary. Alternatively, $\Gamma(t)$ may be an open curve, parameterized over $[0,1]$ with two endpoints attached to the rotational axis $x_2$. This leads to a genus-0 surface without boundary, as depicted in Fig.~\ref{fig:sketch}. Overall,  we assume that the parameterization $\vec x$ satisfies 
\begin{equation}
\vec x(\rho,t)\cdot\vec e_1>0\qquad\forall\rho\in\bar{\bI}\setminus\partial\bI\quad\mbox{for all}\quad t\in[0,T]. 
\end{equation}

We assume $|\vec x_\rho|>0$ in $\bar{\mathbb{I}}$ and introduce the arclength $s$ of the curve $\Gamma(t)$, i.e., $\partial_s= |\vec x_\rho|^{-1}\partial_\rho$.  We also define the unit tangent and normal to $\Gamma(t)$ as
\begin{equation}
\vec\tau= \vec x_s = |\vec x_\rho|^{-1}\vec x_\rho,\qquad\vec\nu = -(\vec\tau)^\perp\qquad\rho\in\bar{\mathbb{I}},
\end{equation}
where $(\cdot)^\perp$ represents a clockwise rotation by $\frac{\pi}{2}$. We always choose the sign convention such that $\vec\nu$ points outwards. We denote by $\kappa$ the curvature of the curve $\Gamma(t)$\begin{equation}
\kappa\,\vec\nu = \vec x_{ss},
\end{equation}
so that $\kappa=-1$ if $\Gamma(t)$ is a unit circle.

Due to the axisymmetry, the normal velocity  $\mathcal{V}_{_\mS}$ and the mean curvature $\varkappa_{_\mS}$ of the surface $\mS(t)$ are independent of the azimuthal angle $\theta$. Thus it is natural to consider their corresponding variants that are defined on $\bI$. To this end, we introduce $\mathscr{V}(\rho,t) = \mathcal{V}_{_\mS}\circ \Bigl(\begin{matrix}\vec x(\rho,t)\\ 0\end{matrix}\Bigr)$ and  $\varkappa(\rho,t) = (\varkappa_{_\mS})\circ\Bigl(\begin{matrix}\vec x(\rho,t)\\ 0\end{matrix}\Bigr)$, and it holds that
\begin{equation}\label{eq:vkappa}
\mathscr{V} = \vec x_t\cdot\vec\nu,\qquad \varkappa = \kappa  +  \zeta= \kappa - \frac{\vec\nu\cdot\vec e_1}{\vec x\cdot\vec e_1},
\end{equation}
where $\kappa(\rho,t)$ and $\zeta(\rho,t)$ are the two principal curvatures in the in-plane and azimuthal directions, respectively. 

 We have the following lemma for the time derivative of the curvature quantities.
\begin{lem}\label{lem:timecur} It holds that
\begin{subequations}
\begin{align}
\label{eq:kappa_t}
\kappa_t &= \mV_{ss} + \mV\,\kappa^2 + (\vec x_t\cdot\vec\tau)\,\kappa_s,\\
\zeta_t & = \frac{\vec\tau\cdot\vec\ek_1}{\vec x\cdot\vec\ek_1}\mV_s+\mV\,\zeta^2 + (\vec x_t\cdot\vec\tau)\,\zeta_s,
\label{eq:zeta_t}
\end{align}
which implies that 
\begin{equation}
\label{eq:varkappa_t}
\varkappa_t = \frac{1}{\vec x\cdot\vec e_1}\left(\vec x\cdot\vec e_1\,\mathscr{V}_s\right)_s +\mathscr{V}\left(\varkappa^2+\frac{2\,\kappa\,\vec\nu\cdot\vec e_1}{\vec x\cdot\vec e_1}\right) + (\vec x_t\cdot\vec\tau)\,\varkappa_s.
\end{equation}
\end{subequations}
\end{lem}
\begin{proof}
We note that \eqref{eq:kappa_t} follows directly from \cite[Lemma 39(i)]{Barrett20}, see also \cite[Lemma 2.1]{GNZ25willmore}. So we omit the proof here, and only prove \eqref{eq:zeta_t}.

We have that the following identities hold 
\begin{subequations}
\begin{alignat}{3}
\label{eq:taunus}
\vec\tau_s &= \kappa\,\vec\nu,\qquad &&\vec\nu_s = -\kappa\,\vec\tau,\\
\vec\tau_t & = [\mV_s + (\vec x_t\cdot\vec\tau)\kappa]\,\vec\nu,\qquad  &&\vec\nu_t = -[\mV_s + (\vec x_t\cdot\vec\tau)\,\kappa]\,\vec\tau.
\label{eq:taunut}
\end{alignat}
\end{subequations}
Here \eqref{eq:taunus} are the well-known Frenet–Serret formulas while the first identity in \eqref{eq:taunut} can be derived based on the property $\partial_t\partial_s = \partial_s\partial_t - ([\vec x_t]_s\cdot\vec\tau)\,\partial_s$, see \cite[(2.9)]{GNZ25willmore}. The second identity in \eqref{eq:taunut} can be obtained by applying $(\cdot)^\perp$ on $\vec\tau_t$.

Then direct calculation yields
\begin{align}
\zeta_t &=-\left(\frac{\vec\nu\cdot\vec e_1}{\vec x\cdot\vec e_1}\right)_t =-\frac{\vec\nu_t\cdot\vec e_1}{\vec x\cdot\vec e_1}+\frac{\vec\nu\cdot\vec e_1\,(\vec x_t\cdot\vec e_1)}{(\vec x\cdot\vec e_1)^2}\nn\\
&= \frac{[\mV_s +(\vec x_t\cdot\vec\tau)\,\kappa]\,(\vec\tau\cdot\vec e_1)}{\vec x\cdot\vec e_1} + \frac{\vec\nu\cdot\vec e_1\,[\vec x_t\cdot\vec\nu\,(\vec\nu\cdot\vec e_1) + \vec x_t\cdot\vec\tau\,(\vec\tau\cdot\vec e_1)]}{(\vec x\cdot\vec e_1)^2}\nn\\
&= \frac{\vec\tau\cdot\vec e_1}{\vec x\cdot\vec e_1}\,\mV_s + \mV\,\zeta^2 + (\vec x_t\cdot\vec\tau)\left[\frac{\kappa\,(\vec\tau\cdot\vec e_1)}{\vec x\cdot\vec e_1} + \frac{\vec\nu\cdot\vec e_1\,(\vec\tau\cdot\vec e_1)}{(\vec x\cdot\vec e_1)^2}\right],\label{eq:zetat1}
\end{align}
where in the second equality we recall \eqref{eq:taunut} as well as the decomposition of $\vec x_t = (\vec x_t\cdot\vec\nu)\,\vec\nu + (\vec x_t\cdot\vec\tau)\,\vec\tau$.
Moreover, it holds that
\begin{equation} \label{eq:zetas1}
 \zeta_s = -\left(\frac{\vec\nu\cdot\vec e_1}{\vec x\cdot\vec e_1}\right)_s =-\frac{\vec\nu_s\cdot\vec e_1}{\vec x\cdot\vec e_1}+\frac{\vec\nu\cdot\vec e_1\,(\vec x_s\cdot\vec e_1)}{(\vec x\cdot\vec e_1)^2}
= \frac{\kappa\,(\vec\tau\cdot\vec e_1)}{\vec x\cdot\vec e_1} + \frac{\vec\nu\cdot\vec e_1\,(\vec\tau\cdot\vec e_1)}{(\vec x\cdot\vec e_1)^2},
\end{equation}
where we recall \eqref{eq:taunus}. Then inserting \eqref{eq:zetas1} into \eqref{eq:zetat1} gives rise to \eqref{eq:zeta_t}. 

Finally, combining \eqref{eq:kappa_t} and \eqref{eq:zeta_t} and noting that $\kappa^2+\zeta^2 =  \varkappa^2 - 2\kappa\,\zeta$ yields \eqref{eq:varkappa_t} as claimed. 
\end{proof}

\begin{rem} It is also possible to consider the time derivative of the mean curvature $\varkappa_{_\mS}$ of the surface $\mS(t)$, see \cite[Lemma 39(i)]{Barrett20}. Then applying the axisymmetric reduction to the time derivative of $\varkappa_{_\mS}$ yields \eqref{eq:varkappa_t} as well. 
\end{rem}

\subsection{Axisymmetric Willmore flow}

In the axisymmetric setting, the total energy of the system in \eqref{eq:energyS} can be reduced to 
\begin{equation}\label{eq:energy}
E_{\bkap}(\vec x(t), \varkappa(t)) = \frac{1}{2}\int_{\mS(t)}(\varkappa_{_\mS}-\bkap)^2\rd A = \pi\int_{\mathbb{I}}\vec x\cdot\vec e_1\,(\varkappa-\bkap)^2|\vec x_\rho|\rd\rho,
\end{equation}
where for simplicity we denote $\vec x(t) = \vec x(\rho, t)$ and $\varkappa(t) = \varkappa(\rho,t)$. 
Similarly, the Willmore flow in \eqref{eqn:WillmoreS} can be reduced to an evolution equation in terms of the generating curve as \cite{BGN19asy}
\begin{align}
\vec x_t\cdot\vec\nu   &= - \frac{1}{\vec x\cdot\vec e_1}\left(\vec x\cdot\vec e_1\,\varkappa_s\right)_s - \frac{2\,\kappa\,\vec\nu\cdot\vec e_1}{\vec x\cdot\vec e_1}\,(\varkappa-\bkap) - \frac{1}{2}(\varkappa^2-\bkap^2)\,\varkappa\qquad\mbox{in}\quad \bI\times(0,T],\label{eqn:willmoreA}
\end{align}
together with the boundary conditions
\begin{subequations}\label{eqn:bd}
\begin{alignat}{3}\label{eq:bd1}
\vec x(\rho,t)\cdot\vec e_1 &= 0\qquad && \mbox{on}\quad\partial\bI\times(0,T],\\
\label{eq:bd2}
\vec x_\rho(\rho,t)\cdot\vec e_2 &= 0\qquad && \mbox{on}\quad\partial\bI\times(0,T],\\
\label{eq:bd3}
\varkappa_\rho &= 0\qquad && \mbox{on}\quad\partial\bI\times(0,T],
\end{alignat}
\end{subequations}
where \eqref{eq:bd1} is the attachment condition and \eqref{eq:bd2} and \eqref{eq:bd3} result from the axisymmetry. 

On recalling Lemma \ref{lem:timecur}, any parameterization $\vec x$ that solves \eqref{eqn:willmoreA} with boundary conditions \eqref{eqn:bd} also solves the following equations
\begin{subequations} \label{eqn:willmore}
\begin{align}\label{eq:willmore1}
 \vec x\cdot\vec e_1\, \mathscr{V}  &=  -\left[\vec x\cdot\vec e_1\,\varkappa_s\right]_s - 2\,\kappa\,(\vec\nu\cdot\vec e_1)(\varkappa-\bkap) - \frac{1}{2}(\vec x\cdot\vec e_1)\,(\varkappa^2-\bkap^2)\,\varkappa,\\
 \label{eq:willmore2}
\vec x\cdot\vec e_1\,\varkappa_t  &= \left[(\vec x\cdot\vec e_1)\,\mathscr{V}_s\right]_s+2\,\kappa\,(\vec\nu\cdot\vec e_1)\,\mathscr{V} + \vec x\cdot\vec e_1\mathscr{V}\,\varkappa^2 + \vec x\cdot\vec e_1(\vec x_t\cdot\vec\tau)\,\varkappa_s,\\
\label{eq:willmore3}
\vec x_t\cdot\vec\nu  &= \mathscr{V},\\
\label{eq:willmore4}
\kappa\,\vec\nu  &= \vec x_{ss},
\end{align}
\end{subequations}
in $\mathbb{I}\times(0,T]$ together with boundary conditions \eqref{eqn:bd}. It follows from \eqref{eq:willmore3} that 
\begin{equation}
\mV_\rho  = (\vec x_t\cdot\vec\nu)_\rho= [\vec x_t]_\rho\cdot\vec\nu + \vec x_t\cdot\vec\nu_\rho = [\vec x_\rho]_t\cdot\vec\nu -\kappa\,|\vec x_\rho|\,(\vec x_t\cdot\vec\tau).\label{eq:vrho}
\end{equation}
On recalling from \eqref{eq:bd2} that $\vec\tau\cdot\vec e_2=0$ on 
$\partial\bI$, we see that differentiating \eqref{eq:bd1} and \eqref{eq:bd2}
with respect to time then yields that $\mV_\rho = 0$ on 
$\partial\bI$. Hence it is natural to include the following boundary condition
\begin{equation}\label{eq:bdV}
\mV_\rho = 0\qquad\mbox{on}\quad\partial\bI\times(0,T].
\end{equation}
In order for the system to be complete, we also need the initial parameterization $\vec x(\cdot, 0)$ and the initial curvature $\varkappa(\cdot, 0)$.

\subsection{Weak formulation}
Following the work in \cite{GNZ25willmore}, we have the following lemma for the convective term in \eqref{eq:willmore2}. 
\begin{lem}It holds that for any $\chi\in H^1(\bI)$
\begin{align}\label{eq:anti}
&\int_\bI\vec x\cdot\vec e_1\,(\vec x_t\cdot\vec\tau)\,\varkappa_\rho\,\chi\rd\rho = 
-\frac{1}{2}\int_\bI\left\{\vec x_t\cdot\vec e_1|\vec x_\rho| + \vec x\cdot\vec e_1\,([\vec x_t]_\rho\cdot\vec\tau)\right\}\,(\varkappa-\bkap)\,\chi\rd\rho\\
& -\frac{1}{2}\int_\bI\vec x\cdot\vec e_1\,(\vec x_t\cdot\vec\nu)\,\varkappa\,(\varkappa-\bkap)\,\chi\,|\vec x_\rho|\rd\rho + \frac{1}{2}\int_\bI\vec x\cdot\vec e_1\,(\vec x_t\cdot\vec\tau)\,\left[\varkappa_\rho\,\chi-(\varkappa-\bkap)\chi_\rho\right]\rd\rho.\nn
\end{align}
\end{lem}
\begin{proof}
The following identity holds for $\chi\in H^1(\bI)$
\begin{equation}
\label{eq:anti1}
\varkappa_\rho\,\chi = (\varkappa-\bkap)_\rho\,\chi = \frac{1}{2}\left[\varkappa_\rho\,\chi - (\varkappa-\bkap)\,\chi_\rho\right]+\frac{1}{2}[(\varkappa-\bkap)\,\chi]_\rho.
\end{equation}
Then we have 
\begin{align}
&\int_\bI\vec x\cdot\vec e_1\,(\vec x_t\cdot\vec\tau)\,\varkappa_\rho\,\chi\rd \rho  \nn\\
&= \frac{1}{2}\int_\bI\vec x\cdot\vec e_1\,(\vec x_t\cdot\vec\tau)\left\{\varkappa_\rho\chi - (\varkappa-\bkap)\,\chi_\rho + [(\varkappa-\bkap)
\chi]_\rho\right\}\rd\rho \nn\\
&=\frac{1}{2}\int_\bI\vec x\cdot\vec e_1\,(\vec x_t\cdot\vec\tau)[\varkappa_\rho\chi - (\varkappa-\bkap)\,\chi_\rho]\rd\rho -\frac{1}{2}\int_\bI\left[\vec x\cdot\vec e_1\,(\vec x_t\cdot\vec\tau)\right]_\rho\,(\varkappa-\bkap)
\chi\rd\rho,
\label{eq:anti2}
\end{align}
where for the last equality, we have used integration by parts and recalled the boundary condition \eqref{eq:bd1} on $\partial\bI$. Next we aim to compute $\left[\vec x\cdot\vec e_1\,(\vec x_t\cdot\vec\tau)\right]_\rho$ in the last term of \eqref{eq:anti2}. We recall the identity $\vec\tau_\rho  = \kappa\,\vec\nu\,|\vec x_\rho|$ in \eqref{eq:taunus} and obtain
\begin{align}
&\left[\vec x\cdot\vec e_1\,(\vec x_t\cdot\vec\tau)\right]_\rho =\vec x_\rho\cdot\vec e_1\,(\vec x_t\cdot\vec\tau)+\vec x\cdot\vec e_1\,([\vec x_t]_\rho\cdot\vec\tau)+\vec x\cdot\vec e_1\,(\vec x_t\cdot\vec\nu)\,\kappa\,|\vec x_\rho| \nn\\
&\quad = (\vec\tau\cdot\vec e_1)(\vec x_t\cdot\vec\tau)|\vec x_\rho| + \vec x\cdot\vec e_1\,([\vec x_t]_\rho)\cdot\vec\tau)+ \vec x\cdot\vec e_1\,(\vec x_t\cdot\vec\nu)\left(\varkappa+\frac{\vec\nu\cdot\vec e_1}{\vec x\cdot\vec e_1}\right)|\vec x_\rho|\nn\\
&\quad =\vec x_t\cdot\vec e_1|\vec x_\rho| + \vec x\cdot\vec e_1\,([\vec x_t]_\rho\cdot\vec\tau) + \vec x\cdot\vec e_1\,(\vec x_t\cdot\vec\nu)\,\varkappa\,|\vec x_\rho|,\label{eq:anti3}
 \end{align}
where we used \eqref{eq:vkappa} for the second equality and the last equality results from the fact that $(\vec\tau\cdot\vec e_1)(\vec x_t\cdot\vec\tau) + (\vec\nu\cdot\vec e_1)(\vec x_t\cdot\vec\nu) = \vec x_t\cdot\vec e_1$.

Now inserting \eqref{eq:anti3} into \eqref{eq:anti2} yields the desired result \eqref{eq:anti}. 
\end{proof}

In order to introduce the weak formulation, we define the function space
\begin{subequations}
\begin{align}
\pV= \bigl\{\vec\eta\in[H^1(\bI)]^2: \vec\eta(\rho)\cdot\vec e_1 = 0\quad\forall\rho\in\partial\bI\bigr\}.
\end{align}
\end{subequations}
We also introduce the $L^2$-inner product on $\bI$ as $(\cdot, \cdot)$. Then we propose our novel weak formulation of \eqref{eqn:willmore} as follows. Given the initial parameterization $\vec x(0)\in \pV$ and the curvature $\varkappa(0)$, for $t\in(0,T]$, we find $\mV(t)\in H^1(\bI)$, $\varkappa(t)\in H^1(\bI)$, $\vec x(t)\in\pV$ and $\kappa(t)\in H^1(\bI)$ such that
\begin{subequations}\label{eqn:weak}
\begin{align}\label{eq:weak1}
&\ipd{\vec x\cdot\vec e_1\,\mathscr{V},~\varphi\,|\vec x_\rho|}-\ipd{\vec x\cdot\vec e_1\,\varkappa_\rho,~\varphi_\rho\,|\vec x_\rho|^{-1}}+2\ipd{\vec\nu\cdot\vec e_1\,\kappa\,(\varkappa-\bkap),~\varphi\,|\vec x_\rho|}\nn\\
&\qquad+\frac{1}{2}\ipd{\vec x\cdot\vec e_1\,(\varkappa^2-\bkap^2)\varkappa,~\varphi\,|\vec x_\rho|}=0\qquad\forall\varphi\in H^1(\bI),\\[0.5em]
\label{eq:weak2}
&\ipd{\vec x\cdot\vec e_1\,\varkappa_t,~\chi\,|\vec x_\rho|}+\frac{1}{2}\ipd{\vec x_t\cdot\vec e_1\, |\vec x_\rho|+ \vec x\cdot\vec e_1\,([\vec x_t]_\rho\cdot\vec\tau),~(\varkappa-\bkap)\,\chi}\nn\\
&\qquad+\ipd{\vec x\cdot\vec e_1\,\mathscr{V}_\rho,~\chi_\rho\,|\vec x_\rho|^{-1}} - 2\ipd{\vec\nu\cdot\vec e_1\,\kappa\,\mV,~\chi\,|\vec x_\rho|}-\frac{1}{2}\ipd{\vec x\cdot\vec e_1\,(\varkappa+\bkap)\,\varkappa\,\mathscr{V},~\chi\,|\vec x_\rho|}\nn\\
&\qquad-\frac{1}{2}\ipd{\vec x\cdot\vec e_1\,(\vec x_t\cdot\vec\tau),~\varkappa_\rho\,\chi - (\varkappa-\bkap)\,\chi_\rho}=0\qquad\forall\chi\in H^1(\bI),\\[0.5em]
\label{eq:weak3}
&\ipd{\vec x_t\cdot\vec\nu,~\xi\,|\vec x_\rho|} - \ipd{\mV,~\xi\,|\vec x_\rho|}=0\qquad\forall\xi\in H^1(\bI),\\[0.5em]
&\ipd{\kappa\,\vec\nu,~\vec\eta\,|\vec x_\rho|} +\ipd{\vec x_\rho,~\vec\eta_\rho\,|\vec x_\rho|^{-1}}=0\qquad\forall\vec\eta\in\pV.
\label{eq:weak4}
\end{align}
\end{subequations}
Here \eqref{eq:weak1} is obtained by taking the inner product of \eqref{eq:willmore1} with $\varphi\,|\vec x_\rho|$ for $\varphi\in H^1(\bI)$, then performing integration by parts and noting the boundary condition \eqref{eq:bd3}. Similarly, we take the inner product of \eqref{eq:willmore2} with $\chi\,|\vec x_\rho|$ for $\chi\in H^1(\bI)$, recall \eqref{eq:anti} and use integration by parts together with the boundary condition \eqref{eq:bdV} to obtain \eqref{eq:weak2}. The equation \eqref{eq:weak3} is due to \eqref{eq:willmore3}, and \eqref{eq:weak4} results from \eqref{eq:willmore4}, the boundary condition in \eqref{eq:bd2} and the fact that $\vec\eta\in \pV$. 

We note here that the essential boundary condition \eqref{eq:bd1} is strongly enforced in the corresponding function space, while  \eqref{eq:weak4} weakly enforces  \eqref{eq:bd2}. However, due to the degenerate weight $\vec x\cdot\vec e_1$ on $\partial\bI$, it is not obvious that \eqref{eq:bd3} and \eqref{eq:bdV} are weakly enforced in \eqref{eq:weak1} and \eqref{eq:weak2}. Nevertheless, a rigorous proof can be made in a similar manner to \cite[Appendix A]{BGN19asy}.  

\begin{rem}We observe that equation \eqref{eq:weak4} serves a dual purpose. The first is to evaluate the curvature $\kappa$ of the generating curve $\Gamma(t)$, which is necessary in both \eqref{eq:weak1} and \eqref{eq:weak2}. The second is that a suitable discretization of \eqref{eq:weak4} yields the BGN tangential velocity, which leads to an equidistribution
property for the evolving curve on the semidiscrete level. 
\end{rem}

We have the following theorem for the energy dissipation law within the weak formulation. 
\begin{thm} A solution of \eqref{eqn:weak} satisfies 
\begin{equation}\label{eq:weakenergyD}
\ddt E_{\bkap}\left(\vec x(t),\varkappa(t)\right)+2\pi\,\ipd{\vec x\cdot\vec e_1\,\mathscr{V}^2,~\,|\vec x_\rho|}=0\qquad\forall t\in(0,T].
\end{equation}
\end{thm}
\begin{proof}
We set $\varphi = \mathscr{V}$ in \eqref{eq:weak1} and $\chi = \varkappa-\bkap$ in \eqref{eq:weak2} to obtain that
\begin{align}
0&=\ipd{\vec x\cdot\vec e_1\,\varkappa_t,~(\varkappa-\bkap)\,|\vec x_\rho|} + \frac{1}{2}\ipd{\vec x_t\cdot\vec e_1\,|\vec x_\rho| + \vec x\cdot\vec e_1\,([\vec x_t]_\rho\cdot\vec\tau),~(\varkappa-\bkap)^2} + \ipd{\vec x\cdot\vec e_1\,\mathscr{V},~\mathscr{V}\,|\vec x_\rho|}\nn\\
&=\dfrac{1}{2}\ddt\ipd{\vec x\cdot\vec e_1\,|\vec x_\rho|,~(\varkappa-\bkap)^2} + \ipd{\vec x\cdot\vec e_1\,\mathscr{V}^2,~|\vec x_\rho|},\nn
\end{align}
which yields \eqref{eq:weakenergyD} as claimed, on recalling \eqref{eq:energy}. 
\end{proof}

\section{Semidiscete scheme}
\label{sec:semi}
 
We consider a uniform partition of the reference domain $\bI$ into
\[\bI = \bigcup_{j=1}^J\bI_j=\bigcup_{j=1}^J[\rho_{j-1},~\rho_j]\quad\mbox{with}\quad\rho_j = j\,h,\quad h=\frac{1}{J},\]
and introduce the finite element spaces 
\begin{subequations}
\begin{align*}
V^h : = \bigl\{\chi\in C^0(\bI): \chi|_{\bI_j}\quad\mbox{is affine}\quad j = 1,\cdots, J\bigr\}\subset H^1(\bI),\qquad \pV^h:=[V^h]^2\cap\pV.
\end{align*}
\end{subequations}
We also define the mass-lumped $L^2$-inner product $(\cdot,\cdot)^h$ as
\begin{equation}
\ipd {u, v }^h = \frac{h}{2}\sum_{j=1}^J 
\left[(u\cdot v)(\rho_j^-) + (u\cdot v)(\rho_{j-1}^+)\right],\label{eq:masslumped}
\end{equation}
for scalar or vector valued functions $u,v$, which are piecewise continuous 
with possible jumps at the nodes $\{\rho_j\}_{j=1}^J$, and 
$u(\rho_j^\pm)=\underset{\delta\searrow 0}{\lim}\ u(\rho_j\pm\delta)$.

Let $(\vec X^h(t))_{t\in[0,T]}$ with $\vec X^h(t)\in \pV^h$ be an approximation to $(\vec x(t))_{t\in[0,T]}$. We also define $\Gamma^h(t) = \vec X^h(\mathbb{I},t)$ and assume that for all $t\in[0,T]$
\begin{equation*}
\vec X^h(\rho,t)\cdot\vec e_1>0,\quad\mbox{in}\quad\bar{\bI}\setminus\partial\bI; \qquad|\vec X_\rho^h|>0\quad\mbox{in}\quad\bar{\bI}.\nn
\end{equation*}
We also set 
\begin{equation*}
\vec\tau^h = \vec X^h_\rho\,|\vec X_\rho^h|^{-1},\qquad\vec\nu^h = -(\vec\tau^h)^\perp. 
\end{equation*}

Now we are ready to introduce the semidiscrete finite element approximation of the weak formulation \eqref{eqn:weak} as follows. Given the initial parameterization $\vec X^h(0)\in \pV^h$ with curvature $\varkappa^h(0)\in V^h$, for each $t\in(0,T]$, we seek for $\mV^h(t)\in V^h$, $\varkappa^h(t)\in V^h$, $\vec X^h(t)\in\pV^h$  and $\kappa^h(t)\in V^h$ such that
\begin{subequations}\label{eqn:semi}
\begin{align}\label{eq:semi1}
&\ipd{\vec X^h\cdot\vec e_1\,\mathscr{V}^h,~\varphi^h\,|\vec X^h_\rho|}-\ipd{\vec X^h\cdot\vec e_1\,\varkappa^h_\rho,~\varphi^h_\rho\,|\vec X^h_\rho|^{-1}}+2\ipd{\vec\nu^h\cdot\vec e_1\,\kappa^h\,(\varkappa^h-\bkap),~\varphi^h\,|\vec X^h_\rho|}\nn\\
&\qquad+\frac{1}{2}\ipd{\vec X^h\cdot\vec e_1\,[(\varkappa^h)^2-\bkap^2]\varkappa^h,~\varphi^h\,|\vec X^h_\rho|}=0\qquad\forall\varphi^h\in V^h,\\[0.5em]
\label{eq:semi2}
&\ipd{\vec X^h\cdot\vec e_1\,\varkappa^h_t,~\chi^h\,|\vec X^h_\rho|}+\frac{1}{2}\ipd{\vec X^h_t\cdot\vec e_1\, |\vec X^h_\rho|+ \vec X^h\cdot\vec e_1\,([\vec X^h_t]_\rho\cdot\vec\tau^h),~(\varkappa^h-\bkap)\,\chi^h}\nn\\
&\qquad+\ipd{\vec X^h\cdot\vec e_1\,\mathscr{V}^h_\rho,~\chi^h_\rho\,|\vec X^h_\rho|^{-1}} - 2\ipd{\vec\nu^h\cdot\vec e_1\,\kappa^h\,\mV^h,~\chi^h\,|\vec X^h_\rho|}\nn\\
&\qquad -\frac{1}{2}\ipd{\vec X^h\cdot\vec e_1\,(\varkappa^h+\bkap)\,\varkappa^h\,\mathscr{V}^h,~\chi^h\,|\vec X^h_\rho|}\nn\\
&\qquad-\frac{1}{2}\ipd{\vec X^h\cdot\vec e_1\,(\vec X^h_t\cdot\vec\tau^h),~\varkappa^h_\rho\,\chi^h - (\varkappa^h-\bkap)\,\chi^h_\rho}=0\qquad\forall\chi^h\in V^h,\\[0.5em]
\label{eq:semi3}
&\ipd{\vec X^h_t\cdot\vec\nu^h,~\xi^h\,|\vec X^h_\rho|}^h - \ipd{\mV^h,~\xi^h\,|\vec X^h_\rho|}=0\qquad\forall\xi^h\in V^h,\\[0.5em]
&\ipd{\kappa^h\,\vec\nu^h,~\vec\eta^h\,|\vec X^h_\rho|}^h +\ipd{\vec X^h_\rho,~\vec\eta^h_\rho\,|\vec X^h_\rho|^{-1}}=0\qquad\forall\vec\eta^h\in\pV^h.
\label{eq:semi4}
\end{align}
\end{subequations}
Here we employ the mass-lumped approximation in \eqref{eq:semi4} in order to lead to the desired BGN tangential motion, and then also in \eqref{eq:semi3} to allow for a well-posedness result on the fully discrete level.

\begin{thm} Let $\bigl(\mV^h(t),~\varkappa^h(t),~\vec X^h(t),~\kappa^h(t)\bigr)$ be a solution to the semidiscrete scheme \eqref{eqn:semi}. Then it holds for all $t\in(0,T]$ that
\begin{equation}\label{eq:semienergyD}
\ddt E_{\bkap}(\vec X^h(t), \varkappa^h(t))=-2\pi\,\ipd{\mV^h,~\mV^h\,|\vec X^h_\rho|}\leq 0.
\end{equation}
\end{thm}
\begin{proof}
Setting $\varphi^h = \mV^h(t)$ in \eqref{eq:semi1} and $\chi^h = \varkappa^h(t) - \bkap$ in \eqref{eq:semi2} yields 
\begin{align}
0&=\ipd{\vec X^h\cdot\vec e_1\,\varkappa^h_t,~(\varkappa^h-\bkap)\,|\vec X^h_\rho|} + \frac{1}{2}\ipd{\vec X^h_t\cdot\vec e_1\,|\vec X^h_\rho| + \vec X^h\cdot\vec e_1\,([\vec X^h_t]_\rho\cdot\vec\tau^h),~(\varkappa^h-\bkap)^2}\nn\\
&\qquad\qquad + \ipd{\vec X^h\cdot\vec e_1\,\mathscr{V}^h,~\mathscr{V}^h\,|\vec X^h_\rho|}\nn\\
&=\dfrac{1}{2}\ddt\ipd{\vec X^h\cdot\vec e_1\,|\vec X^h_\rho|,~(\varkappa^h-\bkap)^2} + \ipd{\vec X^h\cdot\vec e_1\,(\mathscr{V}^h)^2,~|\vec X^h_\rho|},\nn
\end{align}
which implies \eqref{eq:semienergyD} straightforwardly. 
\end{proof}

Denote by
\[\vec a^h_{j-\frac{1}{2}} = \vec X^h(\rho_j, t) -\vec X^h(\rho_{j-1}, t),\qquad j = 1,\ldots, J,\]
the segments of the polygonal curve $\Gamma^h(t)$.
Then we have the following theorem which shows that the semidiscrete solution satisfies an equidistribution property. 

\begin{thm}[equidistribution] \label{thm:equid} Let $(\mathscr{V}^h(t), \varkappa^h(t), \vec X^h(t), \kappa^h(t))$ be a solution to the semidiscrete scheme \eqref{eqn:semi}. For a fixed time $t\in(0,T]$, it holds for $j=1,\ldots, J$ that
\begin{equation}\label{eq:equid}
|\vec a^h_{j-\frac{1}{2}}| = |\vec a^h_{j+\frac{1}{2}}|\quad\mbox{if}\quad \vec a_{j-\frac{1}{2}} \nparallel \vec a_{j+\frac{1}{2}}.
\end{equation}
\end{thm}
\begin{proof}
The desired equidistribution property in \eqref{eq:equid} results from \eqref{eq:semi4}, and the rigorous proof can be found in \cite[Remark 2.4]{BGN07}.
\end{proof}

\section{Fully discrete finite element approximations}
\label{sec:fulld}

We further divide the time interval uniformly by $[0,T]=\bigcup_{m=1}^{M}[t_{m-1}, t_{m}]$, where $t_m= m\ttau$ and  $\ttau = T/M$ is the time step size.  Let $\vec X^m\in \pV^h$ be an approximation to $\vec x(t_m)$ for $0\leq m\leq M$. We define the polygonal curve $\Gamma^m = \vec X^m(\bar{\mathbb{I}})$. Throughout this section we always assume that 
\begin{equation}\label{eq:assump}
 \vec X^m\cdot\vec e_1>0\quad\mbox{in}\quad\bar{\bI}\setminus\partial\bI;\qquad |\vec X_\rho^m|>0\quad\mbox{in}\quad\bar{\bI};\quad  m = 0,1,\ldots, M.
\end{equation}
We also set $\partial_s = |\vec X_\rho^m|^{-1} \partial_\rho$ and
\begin{equation}\label{eq:dtauv}
\vec\tau^m =\vec X_\rho^m\,|\vec X_\rho^m|^{-1},\qquad \vec\nu^m = -(\vec\tau^m)^\perp,
\end{equation}
as the discrete unit tangent and normal vectors of $\Gamma^m$. 

For $m=1,\ldots, M$, we introduce the vertex velocity $\mathscr{\vv V}^m\in \pV^h$ with 
\begin{equation}\label{eq:dvelocity}
\mathscr{\vv V}^m(\rho_j) = \frac{\vec X^m(\rho_j)-\vec X^{m-1}(\rho_j)}{\ttau},\,\qquad j=0,\ldots, J.
\end{equation}

We further introduce $\mathcal{J}^m$ as 
\begin{equation}\label{eq:Jm}
\mathcal{J}^m = \frac{\vec X^{m-1}\cdot\vec e_1\,|\vec X_\rho^{m-1}|}{\vec X^m\cdot\vec e_1\,|\vec X_\rho^m|}.
\end{equation}
Then we have the following lemma for $\mathcal{J}^m$.
\begin{lem}
Assume 
\begin{enumerate}[label=$(\mathbf{A \arabic*})$, ref=$\mathbf{A \arabic*}$] 
\item \label{assumpAI} $\max_{1\leq m\leq M}|\mathscr{\vv V}^m_s|\leq C_1$;
\item \label{assumpAII} $\max_{1\leq m\leq M} \left|\frac{\mathscr{\vv V}^m\cdot\vec e_1}{\vec X^m\cdot\vec e_1}\right|\leq C_2$;
\end{enumerate}
where $C_1,C_2$ are some constants which are  independent of the time step size $\ttau$. 
Then it holds that
\begin{equation}\label{eq:sJm}
\sqrt{\mathcal{J}^m} = 1-\frac{1}{2}\left(\frac{\mathscr{\vv V}_\rho^m\cdot\vec\tau^m}{|\vec X_\rho^m|} + \frac{\mathscr{\vv V}^m\cdot\vec e_1}{\vec X^m\cdot\vec e_1}\right)\ttau + O(\ttau^2).
\end{equation}
\end{lem}

\begin{proof}
Direct computations, on recalling \eqref{eq:Jm}, give rise to  
\begin{align}
(\mathcal{J}^m)^2  &= \frac{(\vec X^{m-1}\cdot\vec e_1)^2}{(\vec X^{m}\cdot\vec e_1)^2}\frac{\vec X_\rho^{m-1}\cdot\vec X_\rho^{m-1}\,}{\vec X_\rho^{m}\cdot\vec X_\rho^{m}} \nn\\
&= \frac{(\vec X^m\cdot\vec e_1-\ttau\mathscr{\vv V}^m\cdot\vec e_1)^2}{(\vec X^m\cdot\vec e_1)^2}\frac{(\vec X^m_\rho - \ttau\mathscr{\vv V}_\rho^m)\cdot(\vec X^m_\rho - \ttau\mathscr{\vv V}_\rho^m)}{\vec X^m_\rho\cdot\vec X^m_\rho} \nn\\
& = \left[1 - 2\ttau\frac{\mathscr{\vv V}^m\cdot\vec e_1}{\vec X^m\cdot\vec e_1} + \ttau^2\frac{(\mathscr{\vv V}^m\cdot\vec e_1)^2}{(\vec X^m\cdot\vec e_1)^2} \right]\left[1-2\ttau\mathscr{\vv V}_\rho^m\cdot\vec\tau^m|\vec X^m_\rho|^{-1} +\ttau^2\,|\mathscr{\vv V}_s^m|^2\right]\nn\\
&= 1-2\bigl[\mathscr{\vv V}^m\cdot\vec e_1\,(\vec X^m\cdot\vec e_1)^{-1}+\mathscr{\vv V}_\rho^m\cdot\vec\tau^m|\vec X_\rho^m|^{-1} \bigr]\,\ttau + O(\ttau^2).
\end{align}
Then applying the Taylor expansion with $(1+\delta x)^{\frac{1}{4}} = 1+ \frac{1}{4}\delta x +O(\delta x^2)$ yields the result in \eqref{eq:sJm}. 
\end{proof}

\subsection{A linear stable approximation}

Let $\mathscr{V}^m, \varkappa^{m}$ and $\kappa^m$
be the numerical approximations of $\mathscr{V}(\cdot,t)$, $\varkappa(\cdot,t)$ and $\kappa(\cdot, t)$ at $t=t_m$.  We propose the following discretized scheme for \eqref{eqn:semi}: Given the initial data $\vec X^0\in\pV^h$ and $\varkappa^0,\kappa^0\in V^h$, we set $\vec X^{-1}=\vec X^0$, so that $J^{0} = 1$, and then for $m=0,\ldots, M-1$, we find  $\mathscr{V}^{m+1}\in V^h$, $\varkappa^{m+1}\in V^h$, $\vec X^{m+1}\in\pV^h$ and $\kappa^{m+1}\in V^h$ such that
\begin{subequations}\label{eqn:fd}
\begin{align}\label{eq:fd1}
&\ipd{\vec X^m\cdot\vec e_1\,\mathscr{V}^{m+1},~\varphi^h\,|\vec X^m_\rho|}-\ipd{\vec X^m\cdot\vec e_1\,\varkappa^{m+1}_\rho,~\varphi^h_\rho\,|\vec X^m_\rho|^{-1}}\nn\\
&\qquad +2\ipd{\vec\nu^m\cdot\vec e_1\,\kappa^m\,(\varkappa^{m+1}-\bkap),~\varphi^h\,|\vec X^m_\rho|}\nn\\
&\qquad+\frac{1}{2}\ipd{\vec X^m\cdot\vec e_1\,(\varkappa^m+\bkap)\,\varkappa^m\,(\varkappa^{m+1}-\bkap),~\varphi^h\,|\vec X^m_\rho|}=0\qquad\forall\varphi^h\in V^h,\\[0.5em]
\label{eq:fd2}
&\ipd{\vec X^m\cdot\vec e_1\,\frac{\varkappa^{m+1}-\bkap-(\varkappa^m-\bkap)\sqrt{\mathcal{J}^m}}{\ttau},~\chi^h\,|\vec X^m_\rho|}^\diamond+\ipd{\vec X^m\cdot\vec e_1\,\mathscr{V}^{m+1}_\rho,~\chi^h_\rho\,|\vec X^m_\rho|^{-1}}\nn\\
&\qquad - 2\ipd{\vec\nu^m\cdot\vec e_1\,\kappa^m\,\mV^{m+1},~\chi^h\,|\vec X^m_\rho|} -\frac{1}{2}\ipd{\vec X^m\cdot\vec e_1\,(\varkappa^m+\bkap)\,\varkappa^m\,\mathscr{V}^{m+1},~\chi^h\,|\vec X^m_\rho|}\nn\\
&\qquad-\frac{1}{2}\ipd{\vec X^m\cdot\vec e_1\,\frac{(\vec X^m-\vec X^{m-1})\cdot\vec\tau^m}{\ttau},~\varkappa^{m+1}_\rho\,\chi^h - (\varkappa^{m+1}-\bkap)\,\chi^h_\rho}=0\qquad\forall\chi^h\in V^h,\\[0.5em]
\label{eq:fd3}
&\ipd{\frac{\vec X^{m+1}-\vec X^m}{\ttau}\cdot\vec\nu^m,~\xi^h\,|\vec X^m_\rho|}^h - \ipd{\mV^{m+1},~\xi^h\,|\vec X^m_\rho|}=0\qquad\forall\xi^h\in V^h,\\[0.5em]
&\ipd{\kappa^{m+1}\,\vec\nu^m,~\vec\eta^h\,|\vec X^h_\rho|}^h +\ipd{\vec X^{m+1}_\rho,~\vec\eta^h_\rho\,|\vec X^m_\rho|^{-1}}=0\qquad\forall\vec\eta^h\in\pV^h,
\label{eq:fd4}
\end{align}
\end{subequations}
where $(\cdot,\cdot)^\diamond$ represents an approximation of the inner product $(\cdot,\cdot)$ using a quadrature rule that is exact for polynomials up to degree three and has non-negative weights. Examples are Gauss quadrature rules with two or more points. 

We next show that \eqref{eq:fd2} is a consistent temporal discretization of \eqref{eq:semi2}. Recalling \eqref{eq:sJm} and applying the Taylor expansion for the first term in \eqref{eq:fd2},  we get
\begin{align}
&\ipd{\vec X^{m}\cdot\vec e_1\,\frac{\varkappa^{m+1}-\bkap - (\varkappa^m-\bkap)\sqrt{\mathcal{J}^m}}{\ttau},~\chi^h\,|\vec X^m_\rho|}^\diamond\nn\\
&= \ipd{\vec X^m\cdot\vec e_1\,\frac{\varkappa^{m+1}-\varkappa^m}{\ttau},~\chi^h\,|\vec X^m_\rho|}^\diamond  + \frac{1}{2}\ipd{\mathscr{\vv V}^m\cdot\vec \ek_1\,|\vec X^m_\rho|,\,(\varkappa^m-\bkap)\chi^h}^\diamond\nn\\
&\qquad\qquad + \frac{1}{2}\ipd{\vec X^m\cdot\vec e_1\,(\mathscr{\vv V}_\rho^m\cdot\vec\tau^m)\,,~(\varkappa^m-\bkap)~\chi^h}^\diamond  + O(\ttau),
\end{align}
which is hence a consistent temporal discretization of the first two terms in \eqref{eq:semi2}. A similar technique was employed by the authors in
\cite{GNZ25willmore}, and also in 
\cite{GNZ24ale} in the context of energy stable fully discrete ALE finite element
approximations for two-phase Navier--Stokes flow.
 
We note that \eqref{eqn:fd} leads to a linear system of equations at each time level, and in particular  the whole system can be decoupled. To be precise, we can first solve for $\mV^{m+1}$ and $\varkappa^{m+1}$ from the linear system of \eqref{eq:fd1} and  \eqref{eq:fd2}, and then solve for $\vec X^{m+1}$ and $\kappa^{m+1}$ from the linear system of \eqref{eq:fd3} and \eqref{eq:fd4}. 

We next aim to show that \eqref{eqn:fd} admits a unique solution. To this end, we follow \cite{BGN07} to define the vertex normal $\vec\omega^m\in [V^h]^2$ such that
\[
\bigl(\vec\omega^m,~\vec\xi^h\,|\vec X_\rho^m|\bigr)^h = \bigl(\vec\nu^m,~\vec\xi^h\,|\vec X_\rho^m|\bigr)\qquad\forall\vec\xi^h\in [V^h]^2.
\]
It is not difficult to show that
\begin{equation*}
\vec\omega^m(\vec\rho_j) = \left\{\begin{array}{ll}
\frac{-[\vec a^m_{j+\frac{1}{2}} - \vec a^m_{j - \frac{1}{2}}]^\perp}{|\vec a^m_{j+\frac{1}{2}}| + |\vec a^m_{j-\frac{1}{2}}|}\qquad &\rho_j\in\bar{\bI}\setminus\partial\bI,\\[0.4em]
\vec\nu^m(\rho_j)\qquad &\rho_j\in\partial\bI.
\end{array}\right.
\end{equation*}
where
\begin{equation}
\label{eq:ajm}
\vec a_{j-\frac{1}{2}}^m = \vec X^m(\rho_j) - \vec X^m(\rho_{j-1}),\qquad\rho_j\in\bar{\bI}\setminus\partial\bI.
\end{equation}
Moreover, it holds that
\begin{equation}\label{eq:weightnor}
\bigl(\chi^h,~\vec\omega^m\cdot\vec\xi^h\,|\vec X_\rho^m|\bigr)^h = \bigl(\chi^h,~\vec\nu^m\cdot\vec\xi^h\,|\vec X_\rho^m|\bigr)^h\qquad\forall\chi^h\in V^h,\quad \vec\xi^h\in [V^h]^2.
\end{equation}

\begin{thm}[existence and uniqueness]\label{thm:wellposed} 
Under the assumption \eqref{eq:assump}, there exists a unique solution $\left(\mV^{m+1},~\varkappa^{m+1}\right)$ to the system of \eqref{eq:fd1} and \eqref{eq:fd2}. Moreover, on assuming that
\begin{enumerate}[label=$(\mathbf{B \arabic*})$, ref=$\mathbf{B \arabic*}$] 
\item \label{assumpBI}  $\vec\omega^m(\rho_j)\neq\vec 0$ for all $\rho_j\in\bI$; 
\item \label{assumpBII}  $\vec\omega^m(\rho_j)- (\vec\omega^m(\rho_j)\cdot\vec e_1)\vec e_1\neq \vec 0$ for all $\rho_j\in\partial\bI$;
\item \label{assupBIII} ${\rm dim\; span}
\bigl\{\vec\omega^{m}(\rho_j)\bigr\}_{j=0}^J=2$;
\end{enumerate}
there exists a unique solution $\left(\kappa^{m+1}, \vec X^{m+1}\right)$ to the system of \eqref{eq:fd3} and \eqref{eq:fd4} for a given $\mathscr{V}^{m+1}$. 
\end{thm}

\begin{proof} We first prove that the linear system of \eqref{eq:fd1} and \eqref{eq:fd2} admits a unique solution $(\mV^{m+1},~\varkappa^{m+1})$. It suffices to show that the corresponding homogeneous system has only the zero solution.

Thus we first consider $(\mV,\varkappa)\in V^h\times V^h$ such that
\begin{subequations}\label{eqn:homoVk}
\begin{align}
\label{eq:homo1}
&\ipd{\vec X^m\cdot\vec e_1\,\mathscr{V},~\varphi^h\,|\vec X^m_\rho|}-\ipd{\vec X^m\cdot\vec e_1\,\varkappa_\rho,~\varphi^h_\rho\,|\vec X^m_\rho|^{-1}}+2\ipd{\vec\nu^m\cdot\vec e_1\,\kappa^m\,\varkappa,~\varphi^h\,|\vec X^m_\rho|}\nn\\
&\qquad+\frac{1}{2}\ipd{\vec X^m\cdot\vec e_1\,(\varkappa^m+\bkap)\,\varkappa^m\,\varkappa,~\varphi^h\,|\vec X^m_\rho|}=0\qquad\forall\varphi^h\in V^h,\\[0.5em]
\label{eq:homo2}
&\ipd{\vec X^m\cdot\vec e_1\,\frac{\varkappa}{\ttau},~\chi^h\,|\vec X^m_\rho|}^\diamond+\ipd{\vec X^m\cdot\vec e_1\,\mathscr{V}_\rho,~\chi^h_\rho\,|\vec X^m_\rho|^{-1}}\nn\\
&\qquad - 2\ipd{\vec\nu^m\cdot\vec e_1\,\kappa^m\,\mV,~\chi^h\,|\vec X^m_\rho|} -\frac{1}{2}\ipd{\vec X^m\cdot\vec e_1\,(\varkappa^m+\bkap)\,\varkappa^m\,\mathscr{V},~\chi^h\,|\vec X^m_\rho|}\nn\\
&\qquad-\frac{1}{2}\ipd{\vec X^m\cdot\vec e_1\,\frac{(\vec X^m-\vec X^{m-1})\cdot\vec\tau^m}{\ttau},~\varkappa_\rho\,\chi^h - \varkappa\,\chi^h_\rho}=0\qquad\forall\chi^h\in V^h.
\end{align}
\end{subequations}
We then choose $\varphi^h =\ttau \mV$ in \eqref{eq:homo1} and $\chi^h =\ttau\, \varkappa$ in \eqref{eq:homo2} and combine the two equations to get
\begin{equation}
\ttau\,\ipd{\vec X^m\cdot\vec e_1\,\mV,~\mV\,|\vec X^m_\rho|} + \ipd{\vec X^m\cdot\vec e_1\,\varkappa,~\varkappa\,|\vec X_\rho^m|}=0,\nn
\end{equation}
which yields that $\mV = 0$ and $\varkappa=0$ by using \eqref{eq:assump}. Thus \eqref{eq:fd1} and \eqref{eq:fd2} has a unique solution.

In order to prove the well-posedness for the linear system of  \eqref{eq:fd3} and \eqref{eq:fd4}, we also consider its corresponding homogeneous system: find $(\vec X, \kappa)\in\pV^h\times V^h$ such that
\begin{subequations}\label{eqn:homoXk}
\begin{align}
\label{eq:homo3}
&\ipd{\frac{\vec X}{\ttau}\cdot\vec\nu^m,~\xi^h\,|\vec X^m_\rho|}^h =0\qquad\forall\xi^h\in V^h,\\[0.5em]
&\ipd{\kappa\,\vec\nu^m,~\vec\eta^h\,|\vec X^m_\rho|}^h +\ipd{\vec X_\rho,~\vec\eta^h_\rho\,|\vec X^m_\rho|^{-1}}=0\qquad\forall\vec\eta^h\in\pV^h.
\label{eq:homo4}
\end{align}
\end{subequations}
Then setting $\xi^h = \ttau\,\kappa$ in \eqref{eq:homo3} and $\vec\eta^h = \vec X$ in \eqref{eq:homo4} and combining the two equations, we obtain
\begin{equation}
\ipd{\vec X_\rho,~\vec X_\rho\,|\vec X_\rho^m|^{-1}}=0,\nn
\end{equation}
which implies that $\vec X=\vec X^c$ is a constant. Then recalling \eqref{eq:masslumped} and \eqref{eq:weightnor}, we can recast \eqref{eq:homo3} as
\[0 = \ipd{\vec X^c\cdot\vec\nu^m,~\xi^h\,|\vec X_\rho^m|}^h = \ipd{\vec X_c^m\cdot\vec\omega^m,~\xi^h\,|\vec X_\rho^m|}\qquad\forall\xi^h\in V^h.\] 
Thus we have $\vec X^c=0$ in view of the assumption \ref{assupBIII}. Inserting $\vec X^c$ in \eqref{eq:homo4} and recalling \eqref{eq:masslumped} and \eqref{eq:weightnor} gives rise to 
\[0=\ipd{\kappa\,\vec\nu^m,~\vec\eta^h\,|\vec X_\rho^m|}^h = \ipd{\kappa\,\vec\omega^m,~\vec\eta^h\,|\vec X_\rho^m|}^h\qquad\forall\vec\eta^h\in\pV^h.\]
We introduce $\vec\omega_{_\partial}\in \pV^h$ via
\begin{equation}
\vec\omega^m_{_\partial}(\rho_j) = \left\{\begin{array}{ll}
 \vec\omega^m(\rho_j)\qquad&\rho_j\in\bI,\\[0.5em]
\vec\omega^m(\rho_j) - (\vec\omega^m(\rho_j)\cdot\vec e_1)\vec e_1\quad &\rho_j\in\partial\bI,
 \end{array}\right.\nn
\end{equation}
and then choose $\vec\eta^h\in\pV^h$ such that $\vec\eta(\rho_j) = \kappa(\rho_j)\,\vec\omega_\partial(\rho_j)$, for $j=1,\ldots, J$, to obtain that 
\[0=\ipd{\kappa^2\,|\vec\omega_{_\partial}|^2,\,|\vec X_\rho^m|}^h,\]
which  implies that $\kappa = 0$ on recalling the assumptions in \ref{assumpBI} and \ref{assumpBII}. Therefore we have $\vec X=\vec 0$ and $\kappa=0$ for the homogeneous system \eqref{eqn:homoXk}. Thus \eqref{eq:fd3} and \eqref{eq:fd4} admits a unique solution $(\vec X^{m+1}, \varkappa^{m+1})$. 
\end{proof}

\begin{thm}[unconditional stability] Let $\left(\mathscr{V}^{m+1},\varkappa^{m+1},\kappa^{m+1}, \vec X^{m+1}\right)$  be a solution to the linear scheme \eqref{eqn:fd}. Then it holds for $m=0,\cdots, M-1$ 
\begin{equation}\label{eq:stab}
E_{\bkap}(\vec X^m,\varkappa^{m+1}) + 2\pi\,\ttau\,\ipd{\vec X^m\cdot\vec e_1\,(\mathscr{V}^{m+1})^2,\,|\vec X^m_\rho|} \leq  E_{\bkap}(\vec X^{m-1},\varkappa^m).
\end{equation}
\end{thm}
\begin{proof}
We choose $\varphi^h = \ttau\,\mV^{m+1}$ in \eqref{eq:fd1} and $\chi^h = \ttau\varkappa^{m+1}$ in \eqref{eq:fd2}, and then combine these two equations to obtain
\begin{align}
0&=\ttau\,\ipd{\vec X^m\cdot\vec e_1\,\mV^{m+1},~\mV^{m+1}|\vec X_\rho^m|}\nn\\
&\qquad +\ipd{\vec X^m\cdot\vec e_1\,\left[\varkappa^{m+1}-\bkap - (\varkappa^m-\bkap)\sqrt{\mathcal{J}^m}\right],~(\varkappa^{m+1}-\bkap)|\vec X_\rho^m|}^\diamond.\label{eq:stab1}
\end{align}

Using the inequality $(a-b)a\geq \frac{1}{2}(a^2-b^2)$, the fact that the quadrature rules have non-negative weights and are exact for polynomials of degree three,  and recalling \eqref{eq:Jm}, we have
\begin{align}
&\ipd{\vec X^m\cdot\vec e_1\,\left[\varkappa^{m+1}-\bkap - (\varkappa^m-\bkap)\sqrt{\mathcal{J}^m}\right],~(\varkappa^{m+1}-\bkap)|\vec X_\rho^m|}^\diamond\nn\\
&\geq\frac{1}{2}\ipd{\vec X^m\cdot\vec e_1\,|\vec X_\rho^m|,~(\varkappa^{m+1}-\bkap)^2}^\diamond - \frac{1}{2}\ipd{\vec X^m\cdot\vec e_1\,|\vec X_\rho^m|,~(\varkappa^{m}-\bkap)^2\mathcal{J}^m}^\diamond\nn\\
&=\frac{1}{2}\ipd{\vec X^m\cdot\vec e_1\,|\vec X_\rho^m|,~(\varkappa^{m+1}-\bkap)^2} - \frac{1}{2}\ipd{\vec X^{m-1}\cdot\vec e_1\,|\vec X_\rho^{m-1}|,~(\varkappa^{m}-\bkap)^2}\nn\\
& = \frac{1}{2\pi} \left(E_{\bkap}(\vec X^m,\varkappa^{m+1})- E_{\bkap}(\vec X^{m-1}, \varkappa^m)\right).\label{eq:stab2}
\end{align}
Inserting \eqref{eq:stab2} into \eqref{eq:stab1} then leads to \eqref{eq:stab}. 
\end{proof}

\subsection{A nonlinear stable approximation}
We next propose a nonlinear discretization of \eqref{eqn:semi} which will also lead to an unconditional stability estimate.  Given the initial data $\vec X^0\in \pV^h$ and $\varkappa^0,\kappa^0\in V^h$, for $m=0,\ldots, M-1$, we find  $\mathscr{V}^{m+1}\in V^h$, $\varkappa^{m+1}\in V^h$, $\vec X^{m+1}\in\pV^h$ and $\kappa^{m+1}\in V^h$ such that
\begin{subequations}\label{eqn:nfd}
\begin{align}\label{eq:nfd1}
&\ipd{\vec X^m\cdot\vec e_1\,\mathscr{V}^{m+1},~\varphi^h\,|\vec X^m_\rho|}-\ipd{\vec X^m\cdot\vec e_1\,\varkappa^{m+1}_\rho,~\varphi^h_\rho\,|\vec X^m_\rho|^{-1}}\nn\\
&\qquad +2\ipd{\vec\nu^m\cdot\vec e_1\,\kappa^m\,(\varkappa^{m+1}-\bkap),~\varphi^h\,|\vec X^m_\rho|}\nn\\
&\qquad+\frac{1}{2}\ipd{\vec X^m\cdot\vec e_1\,(\varkappa^m+\bkap)\,\varkappa^m\,(\varkappa^{m+1}-\bkap),~\varphi^h\,|\vec X^m_\rho|}=0\qquad\forall\varphi^h\in V^h,\\[0.5em]
\label{eq:nfd2}
&\ipd{\vec X^m\cdot\vec e_1\,\frac{\varkappa^{m+1} - \varkappa^m}{\ttau},~\chi^h\,|\vec X^m_\rho|} +\frac{1}{2}\ipd{\frac{\vec X^{m+1}-\vec X^m}{\ttau}\cdot\vec e_1\,|\vec X_\rho^{m+1}|,~(\varkappa^{m+1}-\bkap)\,\chi^h}\nn\\
&\qquad +\frac{1}{2}\ipd{\vec X^m\cdot\vec e_1\,\frac{(\vec X^{m+1}-\vec X^m)_\rho\cdot\vec X_\rho^{m+1}}{\ttau\,|\vec X_\rho^m|},~(\varkappa^{m+1}-\bkap)\,\chi^h}+ \ipd{\vec X^m\cdot\vec e_1\,\mathscr{V}^{m+1}_\rho,~\chi^h_\rho\,|\vec X^m_\rho|^{-1}} \nn\\
&\qquad - 2\ipd{\vec\nu^m\cdot\vec e_1\,\kappa^m\,\mV^{m+1},~\chi^h\,|\vec X^m_\rho|} -\frac{1}{2}\ipd{\vec X^m\cdot\vec e_1\,(\varkappa^m+\bkap)\,\varkappa^m\,\mathscr{V}^{m+1},~\chi^h\,|\vec X^m_\rho|}\nn\\
&\qquad-\frac{1}{2}\ipd{\vec X^m\cdot\vec e_1\,\frac{(\vec X^{m+1}-\vec X^{m})\cdot\vec\tau^m}{\ttau},~\varkappa^{m+1}_\rho\,\chi^h - (\varkappa^{m+1}-\bkap)\,\chi^h_\rho}=0\quad\forall\chi^h\in V^h,\\[0.5em]
\label{eq:nfd3}
&\ipd{\frac{\vec X^{m+1}-\vec X^m}{\ttau}\cdot\vec\nu^m,~\xi^h\,|\vec X^m_\rho|}^h - \ipd{\mV^{m+1},~\xi^h\,|\vec X^m_\rho|}=0\qquad\forall\xi^h\in V^h,\\[0.5em]
&\ipd{\kappa^{m+1}\,\vec\nu^m,~\vec\eta^h\,|\vec X^h_\rho|}^h +\ipd{\vec X^{m+1}_\rho,~\vec\eta^h_\rho\,|\vec X^m_\rho|^{-1}}=0\qquad\forall\vec\eta^h\in\pV^h. 
\label{eq:nfd4}
\end{align}
\end{subequations}
We note here that the introduced method \eqref{eqn:nfd} leads to a nonlinear system due to the nonlinear terms in \eqref{eq:nfd2}. Nevertheless, we have the following theorem which mimics the energy dissipation law \eqref{eq:weakenergyD} on the fully discrete level. 
\begin{thm}[unconditional stability] Let $\left(\mathscr{V}^{m+1},\varkappa^{m+1},\kappa^{m+1}, \vec X^{m+1}\right)$  be a solution to the nonlinear scheme \eqref{eqn:fd}. Then  it holds for $m=0,\cdots, M-1$
\begin{equation}\label{eq:nstab}
E_{\bkap}(\vec X^{m+1},\varkappa^{m+1}) + 2\pi\,\ttau\,\ipd{\vec X^m\cdot\vec e_1\,(\mathscr{V}^{m+1})^2,\,|\vec X^m_\rho|} \leq  E_{\bkap}(\vec X^m,\varkappa^m).
\end{equation}
\end{thm}
\begin{proof}
We set $\varphi^h = \ttau\,\mathscr{V}^{m+1}$ in \eqref{eq:nfd1} and $\chi^h = \ttau\,(\varkappa^{m+1}-\bkap)$ in \eqref{eq:nfd2} and combine the two equations to obtain that
\begin{align}\label{eq:nstab1}
&0=\ttau\,\ipd{\vec X^m\cdot\vec e_1\,\mathscr{V}^{m+1},\mathscr{V}^{m+1}\,|\vec X_\rho^m|} + \ipd{\vec X^m\cdot\vec e_1\,(\varkappa^{m+1}-\varkappa^m),~(\varkappa^{m+1}-\bkap)|\vec X_\rho^m|}\\
&\qquad+\frac{1}{2}\ipd{(\vec X^{m+1}-\vec X^m)\cdot\vec e_1\,|\vec X_\rho^{m+1}|,~(\varkappa^{m+1}-\bkap)^2}\nn\\
&\qquad+\frac{1}{2}\ipd{\vec X^m\cdot\vec e_1\,\frac{(\vec X^{m+1}-\vec X^m)_\rho\cdot\vec X_\rho^{m+1}}{|\vec X_\rho^m|},~(\varkappa^{m+1}-\bkap)^2}.\nn
\end{align}

Again using the inequality $(a-b)a\geq \frac{1}{2}(a^2-b^2)$ we have
\begin{align}\label{eq:nstab2}
&\ipd{\vec X^m\cdot\vec e_1\,\left[\varkappa^{m+1}-\bkap-(\varkappa^m-\bkap)\right],~(\varkappa^{m+1}-\bkap)|\vec X_\rho^m|}\nn\\
&\geq \frac{1}{2}\ipd{\vec X^m\cdot\vec e_1\,|\vec X^m_\rho|,~(\varkappa^{m+1}-\bkap)^2} - \frac{1}{2}\ipd{\vec X^m\cdot\vec e_1\,|\vec X^m_\rho|,~(\varkappa^{m}-\bkap)^2}\nn\\
&=\frac{1}{2}\ipd{\vec X^m\cdot\vec e_1\,|\vec X^m_\rho|,~(\varkappa^{m+1}-\bkap)^2}  - \frac{1}{2\pi}E_{\bkap}(\vec X^m,\varkappa^m).
\end{align}
Moreover, 
\begin{align}
&\ipd{\vec X^m\cdot\vec e_1\,\frac{(\vec X^{m+1}-\vec X^m)_\rho\cdot\vec X_\rho^{m+1}}{|\vec X_\rho^m|},~(\varkappa^{m+1}-\bkap)^2}\nn\\
&\geq \ipd{\vec X^m\cdot\vec e_1\,(\varkappa^{m+1}-\bkap)^2,~\frac{1}{2}\left(|\vec X_\rho^{m+1}|^2-|\vec X_\rho^m|^2\right)|\vec X_\rho^m|^{-1}}\nn\\
&= \ipd{\vec X^m\cdot\vec e_1\,(\varkappa^{m+1}-\bkap)^2,~\frac{1}{2}\left(|\vec X_\rho^{m+1}|^2|\vec X_\rho^m|^{-2}-1\right)|\vec X_\rho^m|}\nn\\
&\geq \ipd{\vec X^m\cdot\vec e_1\,(\varkappa^{m+1}-\bkap)^2,~|\vec X_\rho^{m+1}| - |\vec X_\rho^m|},
\label{eq:nstab3}
\end{align}
where we used the fact $(a-b)a\geq \frac{1}{2}(a^2-b^2)$ for the first inequality and the fact $\frac{1}{2}(a^2-1)\geq a-1$ for the second inequality. 

Inserting \eqref{eq:nstab2} and \eqref{eq:nstab3} into \eqref{eq:nstab1}, we get
\begin{align}
0&\geq \ttau\,\ipd{\vec X^m\cdot\vec e_1\,\mathscr{V}^{m+1},\mathscr{V}^{m+1}\,|\vec X_\rho^m|} + \frac{1}{2}\ipd{\vec X^m\cdot\vec e_1\,|\vec X^m_\rho|,~(\varkappa^{m+1}-\bkap)^2} - \frac{1}{2\pi}E(\vec X^m,\varkappa^m)\nn\\
&\qquad +\frac{1}{2}\ipd{(\vec X^{m+1}-\vec X^m)\cdot\vec e_1\,|\vec X_\rho^{m+1}|,~(\varkappa^{m+1}-\bkap)^2}\nn\\
&\qquad+\frac{1}{2}\ipd{\vec X^m\cdot\vec e_1\,(\varkappa^{m+1}-\bkap)^2,~|\vec X_\rho^{m+1}| - |\vec X_\rho^m|}\nn\\
&=\ttau\,\ipd{\vec X^m\cdot\vec e_1\,\mathscr{V}^{m+1},\mathscr{V}^{m+1}\,|\vec X_\rho^m|}  + \frac{1}{2\pi}\left(E_{\bkap}(\vec X^{m+1},\varkappa^{m+1})-E_{\bkap}(\vec X^m, \varkappa^m\right),
\end{align}
which implies \eqref{eq:nstab}. 
\end{proof}

\section{Numerical results}
\label{sec:numr}

Throughout the experiments, we start with a nodal interpolation of the initial parameterization $\vec x(\cdot,0)$ to obtain $\vec Y^{0}\in \pV^h$. We then employ the BGN method \cite{BGN07} with zero normal velocity to compute $\delta\vec Y^0\in \pV^h$ and $\kappa^0\in V^h$ such that
\begin{align*}
&\bigl(\vec\nu^{0}_Y\cdot\delta\vec Y^0,~\xi^h\,|\vec Y^0_\rho|\bigr)^h=0 \qquad\forall\xi^h\in V^h,\\
&\bigl(\kappa^{0}\,\vec\nu^0_Y,~\vec\eta^h\,|\vec Y^0_\rho|\bigr) ^h= -\bigl([\vec Y^{0}+\delta\vec Y^0]_\rho,~\vec\eta^h_\rho\,|\vec Y^0_\rho|^{-1}\bigr)\qquad\forall\vec\eta^h\in \pV^h, 
\end{align*}
where $\vec\nu^0_Y = -(\vec Y^m_\rho)^\perp\,|\vec Y_\rho^m|^{-1}$. This provides the initial data $\vec X^0=\vec Y^0+\delta\vec Y^0$ and $\kappa^0$. We note that it follows from \eqref{eq:bd1} and \eqref{eq:bd2} that \cite[(2.17)]{BGN19asy} 
\begin{equation}
\lim_{\rho\to\rho_0}\frac{\vec\nu\cdot\vec e_1}{\vec x\cdot\vec e_1} = \lim_{\rho\to\rho_0} \frac{\vec\nu_\rho\cdot\vec e_1}{\vec x_\rho\cdot\vec e_1} = \vec\nu_s(\rho_0,t)\cdot\vec\tau(\rho_0,t) = -\kappa(\rho_0,t),\quad\rho_0\in\partial\bI, t\in[0,T].
\end{equation}
This motivates the choice of discrete initial mean curvature $\varkappa^0\in V^h$ as
\begin{equation}
\varkappa^0(\rho_j) = \left\{\begin{array}{ll}
2\kappa^0(\rho_j)\qquad &\mbox{if}\quad\rho_j\in\partial\bI,\\[0.5em]
 \kappa^0(\rho_j) - \frac{\vec\omega^0(\rho_j)\cdot\vec e_1}{\vec X^0(\rho_j)\cdot\vec e_1}\qquad &\mbox{otherwise},
 \end{array}\right.
\end{equation}
on recalling \eqref{eq:vkappa},  where $\vec\omega^0$ is defined as the vertex normal in \eqref{eq:weightnor}. 

For the linear method \eqref{eqn:fd}, the two linear systems can be solved efficiently with the help of the SparseLU factorizations from the Eigen package \cite{eigenweb}. For the nonlinear method \eqref{eqn:nfd}, we can solve it efficiently with the help of Newton's iteration, or alternatively the following Picard iteration: For each $m\geq 0$, we set $\vec X^{m+1,0}= \vec X^m$, then for each $\ell\geq 0$, we compute $\mathscr{V}^{m+1,\ell+1}\in V^h$, $\varkappa^{m+1,\ell+1}\in V^h$, $\vec X^{m+1,\ell+1}\in\pV^h$ and $\kappa^{m+1,\ell+1}\in V^h$ such that
\begin{subequations}\label{eqn:pd}
\begin{align}\label{eq:pd1}
&\ipd{\vec X^m\cdot\vec e_1\,\mathscr{V}^{m+1,\ell+1},~\varphi^h\,|\vec X^m_\rho|}-\ipd{\vec X^m\cdot\vec e_1\,\varkappa^{m+1,\ell+1}_\rho,~\varphi^h_\rho\,|\vec X^m_\rho|^{-1}}\nn\\
&\qquad +2\ipd{\vec\nu^m\cdot\vec e_1\,\kappa^m\,(\varkappa^{m+1,\ell+1}-\bkap),~\varphi^h\,|\vec X^m_\rho|}\nn\\
&\qquad+\frac{1}{2}\ipd{\vec X^m\cdot\vec e_1\,(\varkappa^m+\bkap)\,\varkappa^m\,(\varkappa^{m+1,\ell+1}-\bkap),~\varphi^h\,|\vec X^m_\rho|}=0,\\[0.5em]
\label{eq:pd2}
&\ipd{\vec X^m\cdot\vec e_1\,\frac{\varkappa^{m+1,\ell+1} - \varkappa^m}{\ttau},~\chi^h\,|\vec X^m_\rho|} +\frac{1}{2}\ipd{\frac{\vec X^{m+1,\ell}-\vec X^m}{\ttau}\cdot\vec e_1\,|\vec X_\rho^{m+1,\ell}|,~(\varkappa^{m+1,\ell+1}-\bkap)\,\chi^h}\nn\\
&\qquad +\frac{1}{2}\ipd{\vec X^m\cdot\vec e_1\,\frac{(\vec X^{m+1,\ell}-\vec X^m)_\rho\cdot\vec X_\rho^{m+1,\ell}}{\ttau\,|\vec X_\rho^m|},~(\varkappa^{m+1,\ell+1}-\bkap)\,\chi^h} \nn\\ &\qquad
+ \ipd{\vec X^m\cdot\vec e_1\,\mathscr{V}^{m+1,\ell+1}_\rho,~\chi^h_\rho\,|\vec X^m_\rho|^{-1}} \nn\\
&\qquad - 2\ipd{\vec\nu^m\cdot\vec e_1\,\kappa^m\,\mV^{m+1,\ell+1},~\chi^h\,|\vec X^m_\rho|} -\frac{1}{2}\ipd{\vec X^m\cdot\vec e_1\,(\varkappa^m+\bkap)\,\varkappa^m\,\mathscr{V}^{m+1,\ell+1},~\chi^h\,|\vec X^m_\rho|}\nn\\
&\qquad-\frac{1}{2}\ipd{\vec X^m\cdot\vec e_1\,\frac{(\vec X^{m+1,\ell}-\vec X^{m})\cdot\vec\tau^m}{\ttau},~\varkappa^{m+1,\ell+1}_\rho\,\chi^h - (\varkappa^{m+1,\ell+1}-\bkap)\,\chi^h_\rho}=0,\\[0.5em]
\label{eq:pd3}
&\ipd{\frac{\vec X^{m+1,\ell+1}-\vec X^m}{\ttau}\cdot\vec\nu^m,~\xi^h\,|\vec X^m_\rho|}^h - \ipd{\mV^{m+1,\ell+1},~\xi^h\,|\vec X^m_\rho|}=0\,,\\[0.5em]
&\ipd{\kappa^{m+1,\ell+1}\,\vec\nu^m,~\vec\eta^h\,|\vec X^h_\rho|}^h +\ipd{\vec X^{m+1,\ell+1}_\rho,~\vec\eta^h_\rho\,|\vec X^m_\rho|^{-1}}=0,
\label{eq:pd4}
\end{align}
\end{subequations}
for all $\bigl(\varphi^h,\chi^h,\xi^h,\vec\eta^h\bigr)\in V^h\times V^h\times V^h\times\pV^h$. This leads to a decoupled linear system which can be solved in a similar manner to the linear method \eqref{eqn:fd}. In practice,  we repeat the above iteration \eqref{eqn:pd} until the following condition holds
\begin{equation*}
\max_{1\leq j\leq J}\left\{|\vec X^{m+1,\ell_0+1}(\rho_j)-\vec X^{m+1,\ell_0}(\rho_j)|,~ |\varkappa^{m+1,\ell_0+1}(\rho_j)-\varkappa^{m+1,\ell_0}(\rho_j)|\right\}\leq {\rm tol},
\end{equation*}
for some $\ell_0>0$, where ${\rm tol}=10^{-10}$ is a chosen tolerance. We then set $\mathscr{V}^{m+1} =\mathscr{V}^{m+1,\ell_0+1}$, and similarly for the variables $\varkappa^{m+1}$, $\vec X^{m+1}$ and $\kappa^{m+1}$. 

In the following we will present a variety of numerical examples for our introduced schemes. To examine the good mesh property of our introduced schemes, we introduce the mesh ratio quantity
\begin{equation*}
{\rm R}^m = \frac{\max_{1\leq j\leq J}|\vec a_{j-\frac{1}{2}}|}{\min_{1\leq j\leq J}|\vec a_{j-\frac{1}{2}}|},\quad m=0,1,\ldots, M, 
\end{equation*}
where $\vec a_{j-\frac{1}{2}}$ is defined as in \eqref{eq:ajm}. In the case when ${\rm R}^m=1$, it means that the vertices on the polygonal curve $\Gamma^m$ are equally distributed. 

\subsection{For genus-0 surfaces}

In this subsection, we focus on the case of genus-0 surfaces, so that $\partial\bI = \{0,1\}$. 

\begin{table}[!htp]
\centering
\caption{[$\partial\bI =\{0,1\}, \bkap=-1$] Errors and experimental orders of convergence (EOC) for an evolving sphere under Willmore flow with $\bkap=-1$ by using the linear scheme \eqref{eqn:fd}(upper panel) and the nonlinear scheme \eqref{eqn:nfd} (lower panel), where $h_0 = 1/32, \ttau_0=0.04$ and $T=1$.}
\label{tb:order}
\begin{tabularx}{0.8\textwidth}{@{}Xcccccc@{}}
\toprule
$(h,\ \ttau)$  
&$\norm{\vec X^h - \vec x}_{\infty}$ &EOC &$\norm{\varkappa^h - \varkappa}_{\infty}$ & EOC &$\norm{E_\bkap^h - E_\bkap}_{\infty}$ & EOC  \\
\midrule
$(h_0, \ttau_0)$
&4.75E-3  &--  &3.53E-2  &--   &9.55E-2 &-- \\
$(\frac{h_0}{2}, \frac{\ttau_0}{2^2})$
&1.20E-3  &1.98  &9.81E-3 &1.85  &2.40E-2 &1.99\\
$(\frac{h_0}{2^2}, \frac{\ttau_0}{2^4})$ 
&3.11E-4  &1.95 &2.53E-3 &1.96  &6.09E-3 &1.98\\
$(\frac{h_0}{2^3}, \frac{\ttau_0}{2^6})$ 
&8.19E-5 &1.93 &6.36E-4 &1.99  &1.52E-3 &2.00 \\
$(\frac{h_0}{2^4}, \frac{\ttau_0}{2^8})$ 
&2.16E-5 &1.92 &1.59E-4 &2.00  &3.81E-4 &2.00 \\
\midrule
$(h_0, \ttau_0)$
&1.52E-2  &--  &2.73E-2  &--  &2.37E-1 &--\\
$(\frac{h_0}{2}, \frac{\ttau_0}{2^2})$
&3.83E-3  &1.99  &7.30E-3 &1.90 &6.36E-2 &1.90  \\
$(\frac{h_0}{2^2}, \frac{\ttau_0}{2^4})$ 
&9.70E-4  &1.98 &1.86E-3 &1.97 &1.62E-2 &1.97\\
$(\frac{h_0}{2^3}, \frac{\ttau_0}{2^6})$ 
&2.46E-4 &1.98 &4.66E-4 &2.00 &4.07E-3 &1.99\\
$(\frac{h_0}{2^4}, \frac{\ttau_0}{2^8})$ 
&6.26E-5 &1.97 &1.17E-4 &1.99 &1.01E-3 &2.01\\
\bottomrule
\end{tabularx}
\end{table}

\noindent
{\bf Example 1}: We begin with a convergence experiment for a sphere that was considered in \cite[page 755]{BGN19asy}. We note that a sphere with radius $r(t)$ satisfying 
\begin{equation}\label{eq:Rtime}
r^\prime(t) = -\frac{\bkap}{r(t)}\left(\frac{2}{r(t)} +\bkap \right), \qquad r(0) = r_0\in\bRplus, 
\end{equation}
is an exact solution to the flow \eqref{eqn:WillmoreS}. In the case of $\bkap\neq 0$, the ordinary differential equation in \eqref{eq:Rtime} can be solved with 
\[r(t) = z(t) - \frac{2}{\bkap},\qquad \varkappa(t) = -\frac{2}{r(t)}\qquad\mbox{with}\quad z_0 =r_0+ \frac{2}{\bkap},\]
and $z(t)$ satisfies $\frac{1}{2}\left(z^2(t) - z_0^2\right) - \frac{4}{\bkap}(z(t)-z_0) +\frac{4}{\bkap^2}\ln\frac{z(t)}{z_0} + \bkap^2 t = 0$. 

We consider a semicircle of radius $r_0=1$ for the initial generating curve $\Gamma(0)$ and choose $T=1, \bkap=-1$. By \eqref{eq:Rtime}, we know that the unit sphere will expand towards a sphere of radius 2 in time.  Initially, we set $\vec X^0\in\pV^h$ with
\begin{equation}
\vec X^0(\rho_j) = r_0\left(\begin{matrix}\cos[(\frac{1}{2}-\rho_j)\pi +\epsilon\cos((\frac{1}{2}-\rho_j)\pi)] \\[0.4em] \sin[(\frac{1}{2}-\rho_j)\pi +\epsilon\cos((\frac{1}{2}-\rho_j)\pi)]
\end{matrix}\right),\nn
\end{equation}
where we choose $\epsilon=0.1$ to force a nonuniform distribution of vertices. We introduce the errors of the generating curve, the mean curvature and the energy over the time interval $[0,1]$ as
\begin{subequations}\label{eqn:dderror}
\begin{align}
\label{eq:Xerror}
\norm{\vec X^h-\vec x}_{\infty} &= \max_{1\leq m\leq M}\max_{1\leq j\leq J}\left||\vec X^m(\rho_j)| - r(t_m)\right|,\\
\label{eq:kaerror}
\norm{\varkappa^h-\varkappa}_{\infty} &= \max_{1\leq m\leq M}\max_{1\leq j\leq J}\left|\varkappa^m(\rho_j) - \varkappa(t_m)\right|,\\
\norm{E_{\bkap}^h- E_\bkap}_{\infty} &= \max_{1\leq m\leq M}| E_\bkap^m-E_\bkap(t_m)|,
\label{eq:energyerror}
\end{align}
\end{subequations}
where the discrete energy $E_\bkap^m$ is given by $E_\bkap(\vec X^{m-1},\varkappa^{m})$ for the linear scheme \eqref{eqn:fd} and by $E_\bkap(\vec X^m,\varkappa^{m})$ for the nonlinear scheme \eqref{eqn:nfd}. For the evolving sphere, we also note that $E_\bkap(t) = 2\pi\,(2-r(t))^2$.

The numerical errors and orders of convergence are reported in Table \ref{tb:order}, where we  observe a second-order convergence rate for both the linear scheme \eqref{eqn:fd} and the nonlinear scheme \eqref{eqn:nfd} when the time step is chosen as $\ttau = O(h^2)$.

In fact,  the numerical results from the linear scheme \eqref{eqn:fd}  and the nonlinear scheme
\eqref{eqn:nfd} are graphically indistinguishable, so that in the following we will only report the numerical results from the former.

\noindent
{\bf Example 2}: In this example, we conduct experiments for an initial disc shape of total dimension $7\times 1\times 7$. We first consider the case of $\bkap=0$ and the numerical results are shown in Fig.~\ref{fig:disk1}. As expected, we observe that the flat disk evolves towards a sphere as the steady state. We also report the results of the discrete energy and the mesh ratio ${\rm R}^m$ in Fig.~\ref{fig:disk1ER}, and we observe the energy decay and the equidistribution property. In particular, the discrete energy at the final time $t=10$ is given by $25.27$. This approximates the value $8\pi=25.13$, i.e., the Willmore energy of a sphere.

We next consider the case of $\bkap=-1.25$ and repeat the above experiment. The numerical results are reported in Fig.~\ref{fig:disk2} and Fig.~\ref{fig:disk2ER}. Here we observe that the disk evolves towards a sphere of radius around $-\frac{2}{\bkap}=1.6$, and the energy decreases to zero in time. Moreover, a good mesh quality for the polygonal curve is observed as well for this experiment. 

\begin{figure}[htp]
\centering
\includegraphics[width=0.9\textwidth]{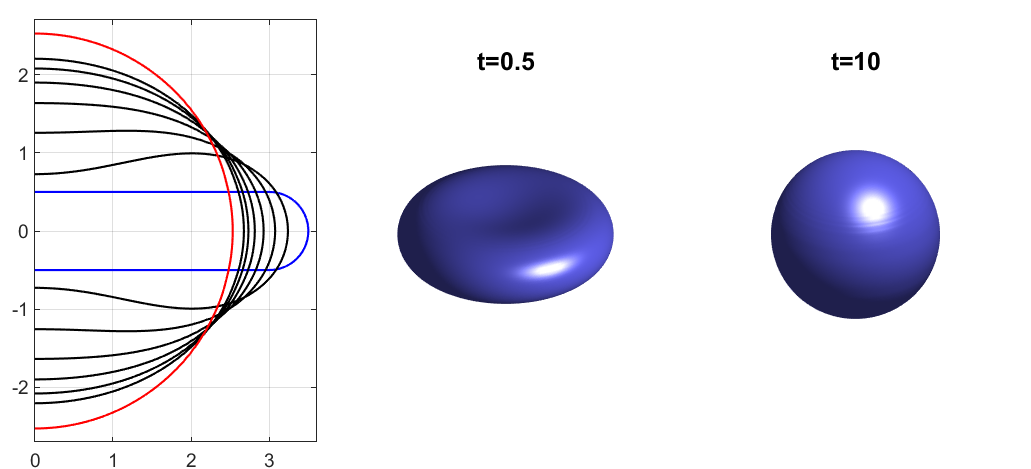}
\caption{[$\bkap=0$, $J = 128$, $\ttau = 10^{-3}$] Evolution of an initial disk of dimension of $7\times 1\times 7$. We plot $\Gamma^m$ at times $t=0,0.5,1,\cdots,3, 10$ and visualize the axisymmetric surfaces at $t=0.5$ and  $t=10$.}
\label{fig:disk1}
\end{figure}

\begin{figure}
\centering
\includegraphics[width=0.95\textwidth]{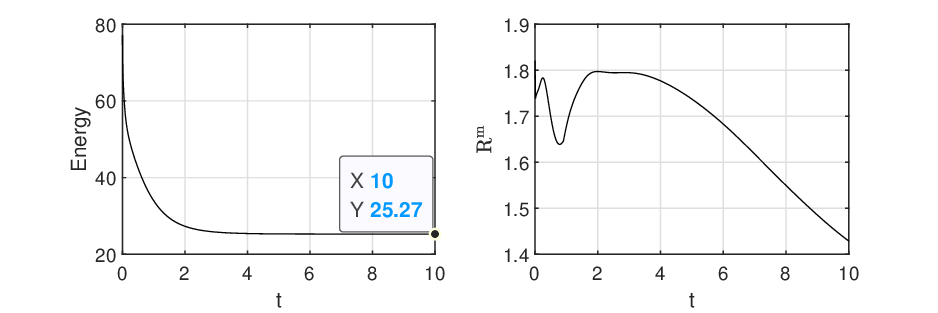}
\caption{Plots of the discrete energy and the mesh ratio ${\rm R}^m$ for the evolution of the disk in Fig.~\ref{fig:disk1}.}
\label{fig:disk1ER}
\end{figure}

\begin{figure}[!htp]
\centering
\includegraphics[width=0.9\textwidth]{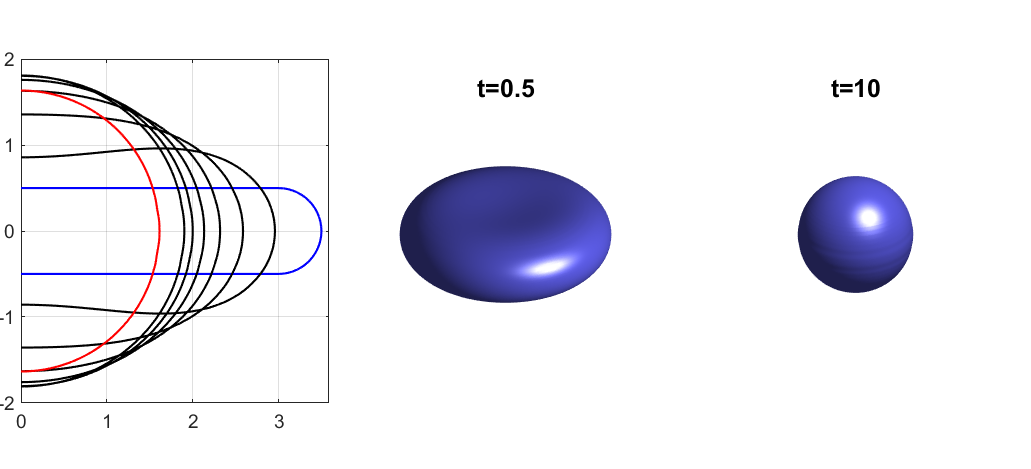}
\caption{[$\bkap=-1.25$, $J = 128$, $\ttau = 10^{-3}$] Evolution of an initial disk of dimension of $7\times 1\times 7$. We plot $\Gamma^m$ at times $t=0,0.5,1,\cdots, 3,10$ and visualize the axisymmetric surfaces at $t=0.5$ and  $t=10$. }
\label{fig:disk2}
\end{figure}

\begin{figure}[!htp]
\centering
\includegraphics[width=0.9\textwidth]{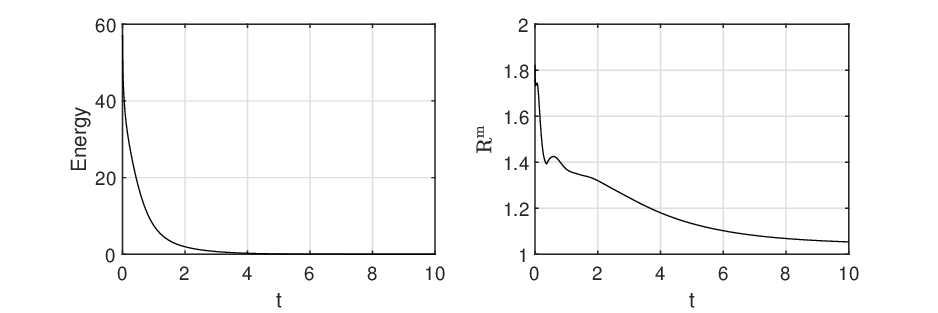}
\caption{Plots of the discrete energy and the mesh ratio ${\rm R}^m$ for the evolution of the disk in Fig.~\ref{fig:disk2}.}
\label{fig:disk2ER}
\end{figure}

\begin{figure}[!htp]
\centering
\includegraphics[width=0.4\textwidth]{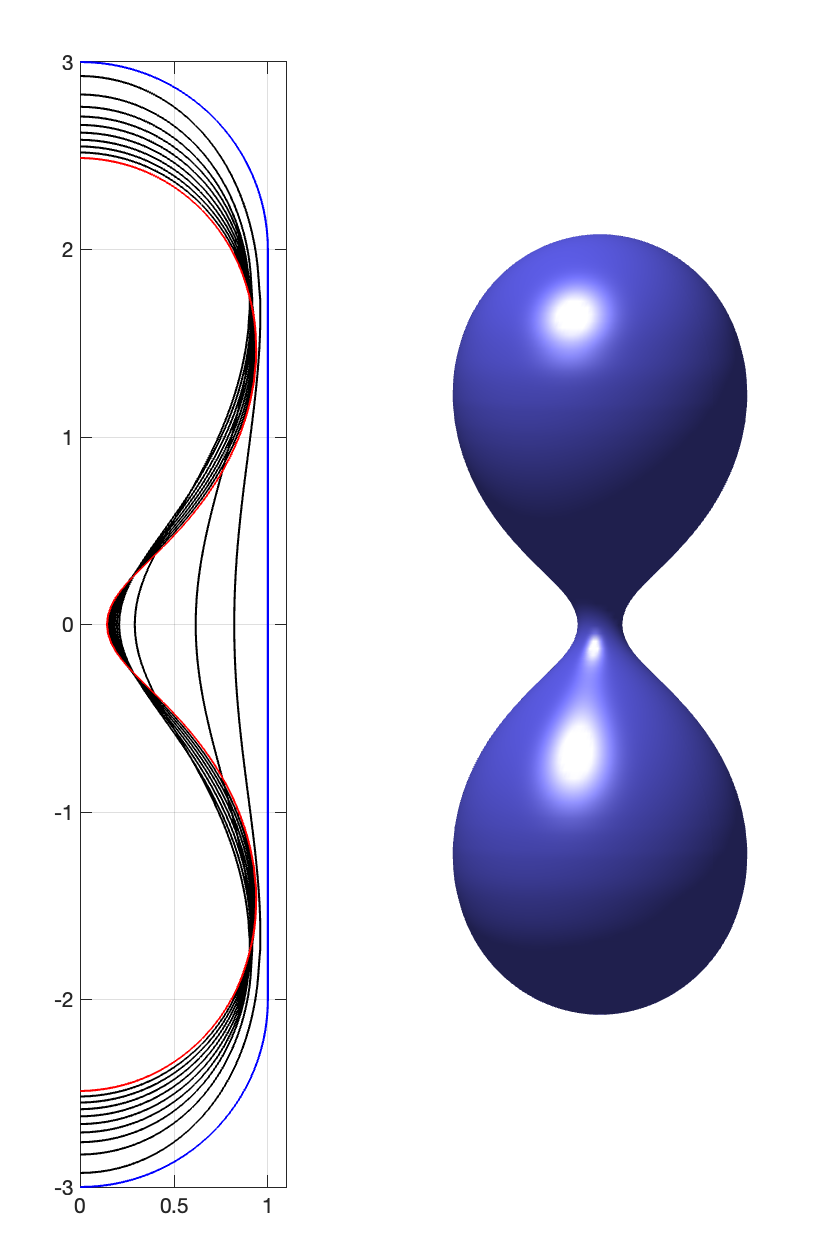}
\includegraphics[width=0.4\textwidth]{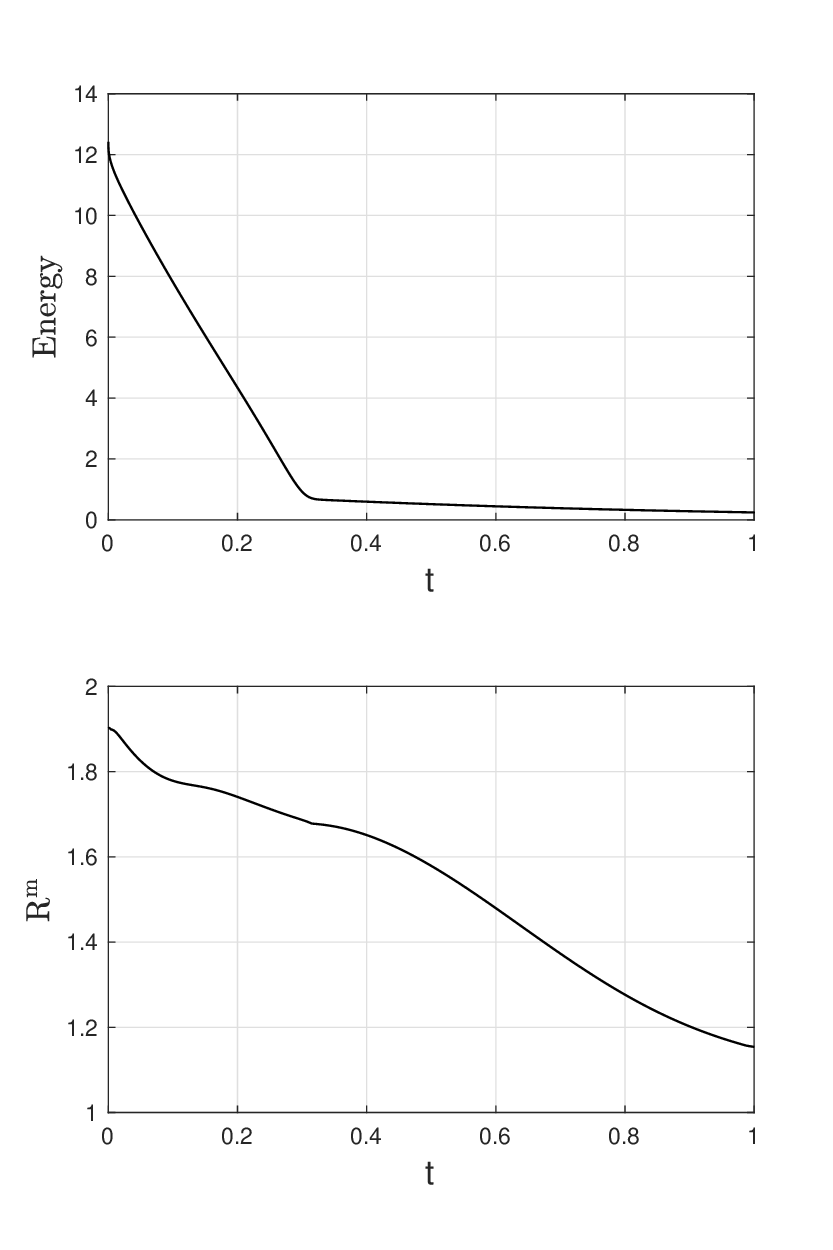}
\caption{[$\bkap=-2, J = 128, \ttau = 2.5\times 10^{-4}$] Evolution of an initial rounded cylinder of dimension $2\times 6\times2$. We plot $\Gamma^m$ at times $t=0, 0.1,0.2,\cdots,1$ and visualize the axisymmetric surface generated by $\Gamma^m$ at time $t=1$. On the right are plots of the discrete energy and the mesh ratio ${\rm R}^m$.}
\label{fig:cylinder1}
\end{figure}

\noindent
{\bf Example 3}: We conduct an experiment for a rounded cylinder of total dimension $2\times 6 \times 2$ in the case of $\bkap=-2$, which was also considered in \cite[Fig.9]{pwfade}. The numerical results are shown in Fig.~\ref{fig:cylinder1}. Here the surface would like to pinch off into two unit spheres, and we compute the solutions until time $t=1$. We observe the energy decay and the good mesh quality during the simulation, even when the vertices on $\Gamma^m$ are approaching the $x_2$–axis.

\subsection{For genus-1 surfaces}
In this subsection, we focus on the case of genus-1 surfaces, so that $\partial\bI = \emptyset$.

\begin{table}[!htp]
\centering
\caption{[$\partial\bI =\emptyset, \bkap=0$] Errors and experimental orders of convergence (EOC) for a Clifford torus under Willmore flow by using the linear scheme \eqref{eqn:fd}, where $h_0 = 1/32, \ttau_0=0.04$ and $T=1$.}
\label{tb:order1}
\begin{tabularx}{0.8\textwidth}{@{}Xcccccc@{}}
\toprule
$(h,\ \ttau)$  
&$\norm{\vec X^h - \vec x}_{\infty}$ &EOC &$\norm{\vec X^h - \vec x}_{L^2}$ & EOC &$\norm{E_\bkap^h - E_\bkap}_{\infty}$ & EOC  \\
\midrule
$(h_0, \ttau_0)$
&8.14E-3  &--  &3.05E-2  &--   &7.77E-2 &-- \\
$(\frac{h_0}{2}, \frac{\ttau_0}{2^2})$
&1.99E-3  &2.03  &7.47E-3 &2.03  &1.92E-2 &2.02\\
$(\frac{h_0}{2^2}, \frac{\ttau_0}{2^4})$ 
&4.98E-4  &2.00 &1.86E-3 &2.01  &4.81E-3 &2.00\\
$(\frac{h_0}{2^3}, \frac{\ttau_0}{2^6})$ 
&1.24E-4 &2.01 &4.63E-4 &2.01  &1.20E-3 &2.00 \\
$(\frac{h_0}{2^4}, \frac{\ttau_0}{2^8})$ 
&3.11E-5 &2.00 &1.15E-4 &2.01  &3.01E-4 &2.00 \\
\bottomrule
\end{tabularx}
\end{table}

\noindent
{\bf Example 4}: We first perform a convergence experiment for a Clifford torus, which is the well-known minimizer for the Willmore energy \eqref{eq:WillmoreE0} among genus-1 surfaces.  Initially we set $\vec X^0\in\pV^h$ such that
\begin{equation}
\vec X^0(\rho_j) = \left(\begin{matrix} R+r\,\cos[(\frac{1}{4}-\rho_j)2\pi +\epsilon\cos((\frac{1}{4}-\rho_j)2\pi)] \\[0.4em] r\,\sin[(\frac{1}{4}-\rho_j)2\pi +\epsilon\cos((\frac{1}{4}-\rho_j)2\pi)]
\end{matrix}\right),\nn
\end{equation}
where $R=\sqrt{2}$ and $r=1$ stand for the major and minor radius of the torus, respectively. We again use $\epsilon=0.1$ to enforce a nonuniform vertex distribution. We note that $R/r = \sqrt{2}$ characterizes a Clifford torus with Willmore energy of $4\pi^2$. In addition to \eqref{eqn:dderror}, we also introduce the manifold distance error \cite{Zhao2021energy}. Let $\Gamma_1$ and $\Gamma_2$ be two closed curves and let $\Omega_1$ and $\Omega_2$ be the region enclosed by $\Gamma_1$ and $\Gamma_2$, respectively. The difference between $\Gamma_1$ and $\Gamma_2$ is measured by the area of the symmetric difference region between $\Omega_1$ and $\Omega_2$
\begin{equation}
{\rm MD}(\Gamma_1,\Gamma_2) = |(\Omega_1\setminus\Omega_2)\cup(\Omega_2\setminus\Omega_1)| = |\Omega_1| + |\Omega_2|-2 |\Omega_1\cap\Omega_2|,\nn 
\end{equation}
where $|\Omega|$ represents the area of the region $\Omega$. We define the error
\begin{equation}
\norm{\vec X^h - \vec x }_{L^2} = {\rm MD }(\Gamma^M, \Gamma(t_M)), \nn 
\end{equation}
to measure the manifold distance errors of the polygonal curve at the final time $t=t_M$. The numerical results are reported in Table \ref{tb:order1}, where we observe again a second-order convergence rate for the numerical solutions.

\begin{figure}[!htp]
\centering
\includegraphics[width=0.9\textwidth]{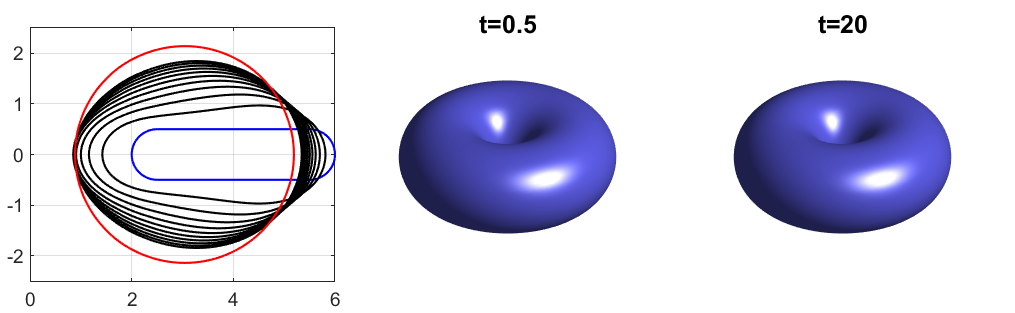}
\includegraphics[width=0.9\textwidth]{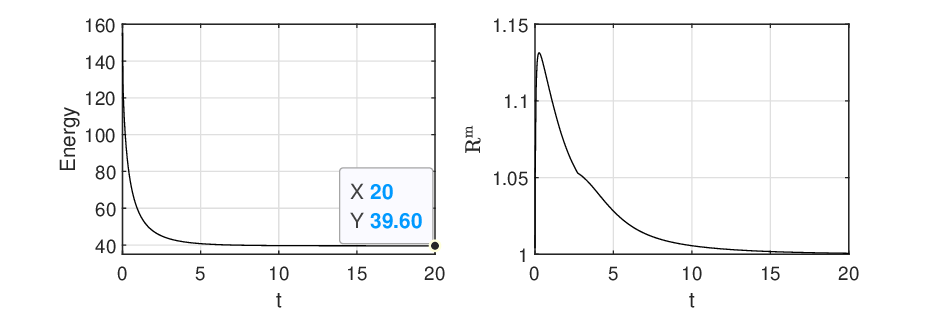}
\caption{[$\bkap=0, J = 128, \ttau = 2.5\times 10^{-4}$] Evolution of an initial flat annulus shape towards a Clifford torus. We plot $\Gamma^m$ at times $t=0,0.5,\cdots,5, 20$ and visualize the axisymmetric surfaces generated by $\Gamma^m$ at $t=0.5$ and $t=20$. On the bottom are plots of the discrete energy and the mesh ratio ${\rm R}^m$.}
\label{fig:torus1}
\end{figure}

\noindent
{\bf Example 5}: We consider the evolution of a genus-1 surface, where the generating curve $\Gamma(0)$ is given by an elongated cigar-like shape of total dimension $4\times 1$, with barycenter $(4,0)$. We first consider the case of $\bkap=0$ and the numerical results are shown in Fig.~\ref{fig:torus1}. We observe that the flat annulus evolves towards a surface of torus shape. 

To further assess the numerical results quantitatively, we introduce the following quantities  
\[\vec X_c^m = \frac{1}{J}\sum_{j=1}^J\vec X^m(\rho_j),\quad r_c^m = \frac{1}{J}\sum_{j=1}^J |\vec X_c^m - \vec X^m(\rho_j)|,\quad d_c^m = \frac{\max_{1\leq j\leq M}|\vec X_c^m - \vec X^m(\rho_j)|}{\min_{1\leq j\leq M}|\vec X_c^m - \vec X^m(\rho_j)|},\]
where $\vec X_c^m$ and $r_c^m$ are the mean centre and radius of the generating curve $\Gamma^m$, respectively, and $d_c^m$ measures the deviation of  $\Gamma^m$ from the circle.  At the final time $t=20$, the mean centre for $\Gamma^m$ is given by $(3.039,0)$ with mean radius $r_c^m = 2.146$ and $d^m_c = 1.007$ which implies the discrete surface is an almost perfect approximation of a torus.  Moreover, the ratio of the two radii of the torus is given by $3.039/2.146=1.416$, which approximates $\sqrt{2}$. In particular, at the final time, the discrete energy is given by $39.60$. This approximates the value of $4\pi^2=39.478$, which is the Willmore energy of the Clifford torus. We notice that our scheme has no problem
to capture this steady state of Willmore flow in a stable and reliable way,
in contrast to the methods from recent works \cite{Ma25energy,LL25axis}. In
fact, our scheme approaches a numerical steady state represented by an
equidistributed polygonal circle.

\begin{figure}[tp]
\centering
\includegraphics[width=0.9\textwidth]{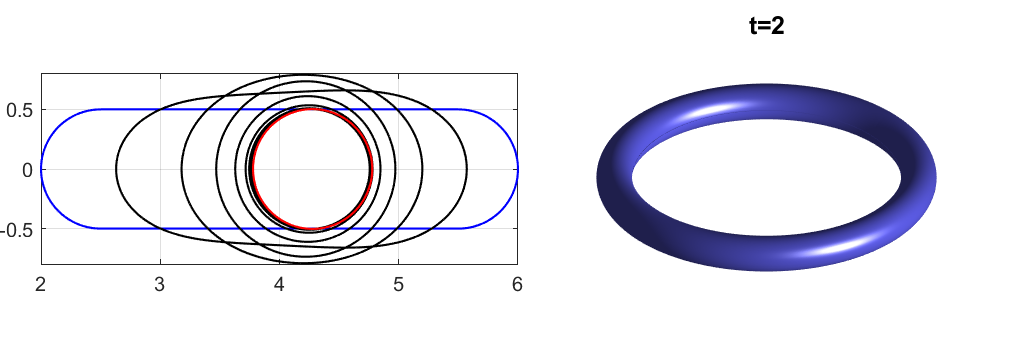}
\includegraphics[width=0.9\textwidth]{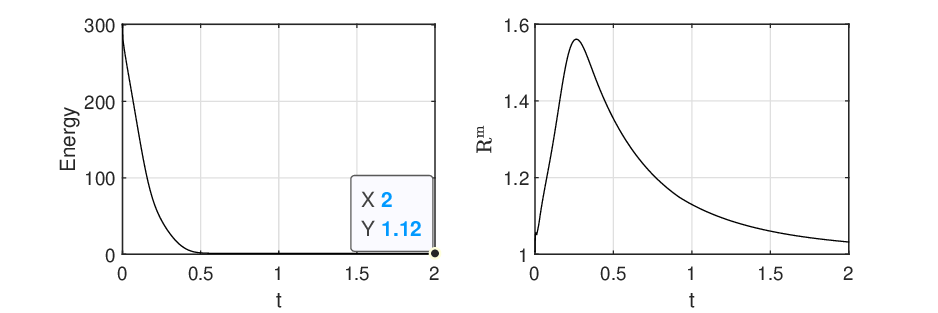}
\caption{[$\bkap=-2, J = 256, \ttau = 6.25\times 10^{-5}$] Evolution of an initial flat annulus shape towards a thin torus. We plot $\Gamma^m$ at times $t=0,0.1,\cdots,1,2$ and visualize the axisymmetric surface generated by $\Gamma^m$ at $t=2$. On the bottom are plots of the discrete energy and the mesh ratio ${\rm R}^m$.}
\label{fig:torus2}
\end{figure}

\begin{figure}[!htp]
\centering
\includegraphics[width=0.9\textwidth]{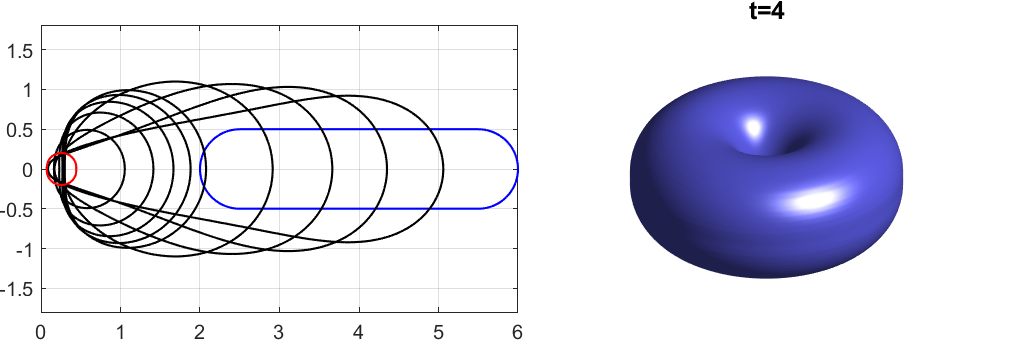}
\includegraphics[width=0.9\textwidth]{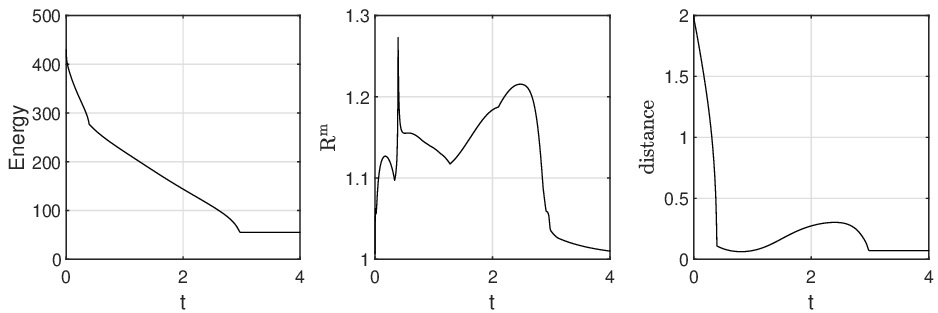}
\caption{[$\bkap=1, J = 256, \ttau = 6.25\times 10^{-5}$] Evolution of an initial flat annulus shape. We plot $\Gamma^m$ at times $t=0,0.5,\cdots,2.5,2.6,2.7,2.8,2.9,3,4$, where the plots for $t=3$ and $t=4$ lie on top of each other.  We also visualize the axisymmetric surface generated by $\Gamma^m$ at $t=4$. On the bottom are plots of the discrete energy, the mesh ratio ${\rm R}^m$ and the distance to the rotational axis.}
\label{fig:torus3}
\end{figure}

\begin{figure}[!htp]
\centering
\includegraphics[width=0.9\textwidth]{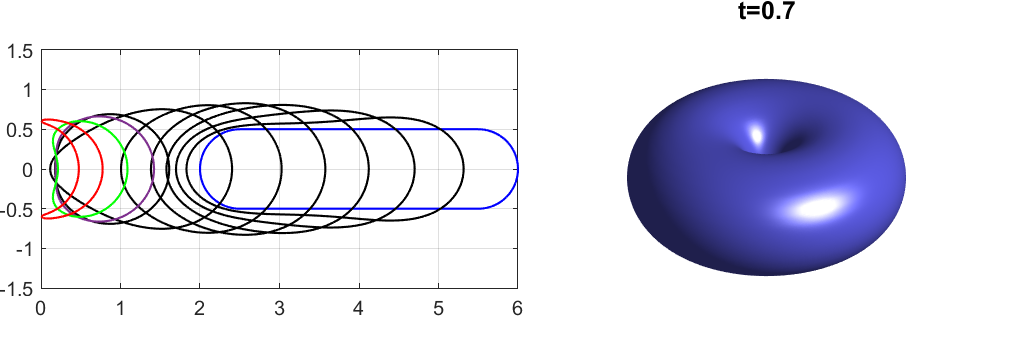}
\includegraphics[width=0.9\textwidth]{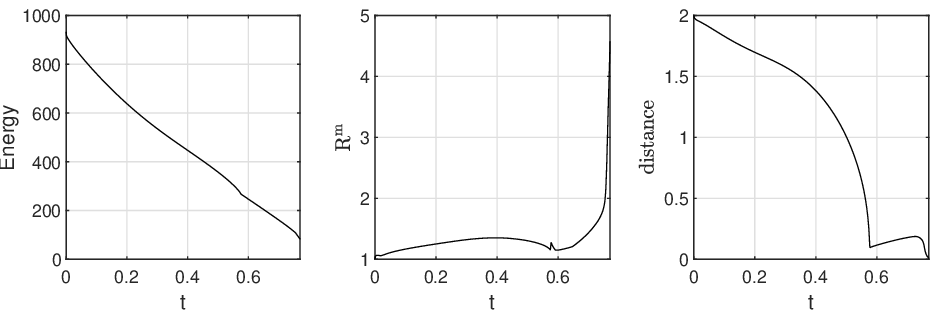}
\caption{[$\bkap=2, J = 128, \ttau = 2.5\times 10^{-4}$] Evolution of an initial flat annulus shape until its pinch-off. We plot $\Gamma^m$ at times $t=0,0.1,\cdots,0.7,0.72,0.75, 0.77$ and visualize the axisymmetric surface generated by $\Gamma^m$ at $t=0.7$. On the bottom are plots of the discrete energy, the mesh ratio ${\rm R}^m$ and the distance to the rotational axis.}
\label{fig:torus4}
\end{figure}

We next repeat the above experiment with $\bkap=-2$, and the numerical results are shown in Fig.~\ref{fig:torus2}. Here the flat annulus evolves towards a much thinner torus shape, in order to reduce its energy. The numerical quantities for $\Gamma^m$ at the final time is given by $\vec X_c^m = (4.277, 0)$ with $r_c^m = 0.501$ and $d_c^m =1.006$. Hence the ratio of the two radii of the torus is given by $R/r = 8.537$.

\noindent
{\bf Example 6}:

In our last example, we consider experiment with the value $\bkap=1$ for the spontaneous curvature, and the numerical results are reported in Fig.~\ref{fig:torus3}. In this case, we observe that the flat annulus shrinks towards the rotational axis to decrease the energy. However, in contrast to some other flows, like, e.g.,  the surface diffusion flow, no pinch-off occurs here. Moreover, the mean centre of the discrete generating curve at the final time is given by $(0.258,0)$ with mean radius $r_c^m=0.196$ and $d_c^m = 1.086$. This means that the corresponding axisymmetirc surface is  also a good approximation of a torus. In fact, an energy decay can be observed as well for the numerical solutions, as well as a good mesh quality. 

Finally, we repeat the experiment with $\bkap=2$ and the numerical results are visualised in Fig.~\ref{fig:torus4}. Here we observe the occurrence of pinch-off, and in particular this leads to a sickle shape for the generating curve. We note that similar sickle shapes have been observed in \cite{Julicher93phase, Seifert91vesicles} which considered the toroidal minimizers of genus-1 shape for the Helfrich energy  with constrained volume and surface area. 

\section{Conclusions}
\label{sec:con}
We proposed and analyzed two parametric finite element approximations for the axisymmetric Willmore flow of surfaces with spontaneous curvature effects. The introduced schemes were shown to satisfy unconditional stability estimates, as well as an asymptotic equidistribution property. This relies on a new geometric PDE for the axisymmetric Willmore flow which combines an evolution of the surface's mean curvature with the curvature formulation of the generating curve. The first curvature approximation enables a stability estimate while the second curvature approximation is responsible for the good mesh quality. We presented a variety of numerical examples to showcase the accuracy, stability and the good mesh quality of the introduced schemes. In the future, we will generalize the idea to Willmore flow of surfaces in the general three-dimensional case.

\section*{Acknowledgements}

This work was partially supported by the National Natural
Science Foundation of China No. 12401572 (Q.Z) and the Key Project of the National Natural Science Foundation of China No.~12494555 (Q.Z).

\footnotesize
\bibliographystyle{abbrv}
\bibliography{bib}

\end{document}